\documentclass[11pt,oneside,reqno]{amsart}
\usepackage{amsfonts}
\usepackage{amssymb,amsmath, amsthm}
\usepackage[numbers]{natbib}
\usepackage{graphicx}
\usepackage{subfig}
\usepackage[pdftex,plainpages=false,colorlinks,hyperindex,bookmarksopen,linkcolor=red,citecolor=blue,urlcolor=blue]{hyperref}
\usepackage{mathrsfs}
\usepackage{xcolor}
\usepackage{epstopdf}
\usepackage{color}
\usepackage{float}
\usepackage{enumerate}
\usepackage{enumitem}
\usepackage{mathtools}
\usepackage{lipsum}
\usepackage{comment}
\usepackage{bbm}
\usepackage[T1]{fontenc}
\usepackage[utf8]{inputenc}
\calclayout
\makeatletter
\def\author@andify{  \nxandlist {\unskip ,\penalty-1 \space\ignorespaces}    {\unskip {} \@@and~}    {\unskip \penalty-2 \space \@@and~}}
\makeatother
\pdfoutput=1
\bibpunct{[}{]}{;}{n}{,}{,}
\newtheorem{theorem}{Theorem}[section]

\newtheorem{coro}{Corollary}[section]
\newtheorem{definition}{Definition}

\theoremstyle{remark}
\newtheorem{remark}{Remark}[section]
\newtheorem{rem}{Remark}[section]
\newtheorem{lem}{Lemma}[section]

\newcommand{\be}{\begin{eqnarray}}
\newcommand{\ee}{\end{eqnarray}}

\numberwithin{equation}{section}
\allowdisplaybreaks

\begin{document}

\title{Hadamard fractional Brownian motion: path properties and Wiener  integration}  

\author[Beghin]{Luisa Beghin$^1$}
\address{${}^1$ Sapienza University of Rome. P.le Aldo Moro 5, 00185, Rome, Italy}
\email{luisa.beghin@uniroma1.it}


\author[De Gregorio]{Alessandro De Gregorio$^2$}
\address{${}^2$ Sapienza University of Rome. P.le Aldo Moro 5, 00185, Rome, Italy}
\email{alessandro.degregorio@uniroma1.it}

\author[Mishura]{Yuliya Mishura$^3$}
\address{${}^3$ Taras Shevchenko National University of Kyiv. Volodymyrska Street 64, 01033 Kyiv, Ukraine}
\email{yuliyamishura@knu.ua}


\begin{abstract}
The so-called Hadamard fractional Brownian motion, as defined in \cite{BEG} by means of Hadamard fractional operators, is a Gaussian process which shares some properties with standard Brownian motion (such as the one-dimensional distribution). However, it also resembles the fractional Brownian motion in many other features as, for instance, self-similarity, long/short memory property, Wiener-integral representation. The logarithmic kernel in the Hadamard fractional Brownian motion represents a very specific and interesting aspect of this process. Our aim here is to analyze some properties of the process' trajectories (i.e. H\"{o}lder continuity, quasi-helix behavior, power variation, local nondeterminism) that are both interesting on their own and serve as a basis for the Wiener integration with respect to it. The respective integration is quite well developed, and the inverse representation is also constructed. We   apply the derived ``multiplicative Sonine  pairs'' to the treatment of the Reproducing Kernel Hilbert Space of the Hadamard fractional Brownian motion, and, as a result, we establish a law of iterated logarithm.
\end{abstract}

\date{}


\maketitle

\section{Introduction}\label{introd}
The paper is devoted to the study of  the  \textit{Hadamard fractional Brownian motion} (hereafter $\mathcal{H}$-fBm, by the first letter of the name of Hadamard). The latter was introduced in \cite{BEG} and defined, in the white-noise space $(\mathcal{S}, \mathcal{B}, \nu)$, as $B^\mathcal{H}_\alpha:=\left\{ B^\mathcal{H}_\alpha (t)\right\}_{t \geq 0}$, with $B^\mathcal{H}_\alpha (t,\omega):=\langle \omega, \thinspace _\mathcal{H}\mathcal{M}^{\alpha} _{-}1_{[0,t)} \rangle$, $t\geq 0,\;\omega \in \mathcal{S%
}^{\prime }(\mathbb{R})$, and
\begin{equation}\label{malf}
\left( \thinspace _\mathcal{H}\mathcal{M}_{-}^{\alpha }f\right) (x):=
\begin{cases}
K_\alpha \left( \thinspace _\mathcal{H}\mathcal{D}_{-}^{(1-\alpha)/2 }f\right) (x),& \alpha \in (0,1),
\\
f(x),& \alpha =1, \\
K_\alpha \left( \thinspace _\mathcal{H}\mathcal{I}_{-}^{(\alpha-1)/2 }f\right) (x),& \alpha \in (1,2),%
\end{cases}%
\end{equation}
for $K_{\alpha}:=\Gamma ((\alpha+1)/2)/\sqrt{\Gamma(\alpha)}$. The operators  $_\mathcal{H}\mathcal{D}_{-}^\beta$ and $_\mathcal{H}\mathcal{I}_{-}^\beta$ in \eqref{malf} are the right-sided Hadamard derivative and integral, respectively (see definitions \eqref{dermin} and \eqref{intmin} below).

It was proved in \cite{BEG} that, for any $\alpha \in (0,1)  \cup (1,2)$, the $\mathcal{H}$-fBm is a centered Gaussian process, with characteristic function given, for $0<t_1<...<t_n$, $n \in \mathbb{N}$, and $k_j \in \mathbb{R},$ $j=1,...,n$, by 
\begin{equation}
    \mathbb{E}  \left(e^{i\sum_{j=1}^{n}k_j B^\mathcal{H}_{\alpha }(t_j)}\right)=\exp \left\{-\frac{1}{2}\sum_{j,l=1}^{n}k_j k_l \sigma^\alpha_{t_j,t_l}   \right\}, \label{cf4}
\end{equation}
where 
\begin{equation} 
\sigma^\alpha_{s,t}:=
\begin{cases}
t  , & s=t, \\
C_\alpha (s \wedge t)\Psi\left(\frac{1-\alpha}{2}, 1-\alpha ; \log\left(\frac{s \vee t }{ s  \wedge t}\right) \right), & s \neq t, \label{sigma}
\end{cases}
\end{equation}
and  $C_\alpha =2^{1-\alpha}\sqrt{\pi}/\Gamma(\alpha/2)$.
The function  
$\Psi\left(a, b ; z\right)$, for $a,b,z \in \mathbb{C}$, $\Re(b) \neq 0, \pm 1, \pm 2,...$, 
is the \textit{Tricomi's confluent hypergeometric function} (see \cite{NIST}, formula (13.2.42)),
which can be represented as 
\begin{equation} \label{eqTricomiINtegral} \Psi\left(a, b ; z\right)=\frac{1}{\Gamma(a)}\int^{\infty}_{0}e^{-sz}s^{a-1}(1+s)^{b-a-1}ds, \end{equation} if $\Re(a)>0,$ $\Re(z) \geq 0$ (see \cite{KIL}, p.\ 30).

It is proved in \cite{BEG} that the process $B^\mathcal{H}_\alpha$ is self-similar with index $1/2$;
moreover, it has non-stationary increments
and the discrete-time increment process 
is anti-persistent for $\alpha \in (0,1)$, while it is long-range dependent for $\alpha \in (1,2)$. 
Finally, the following, finite-dimensional, representation of $\mathcal{H}$-fBm (which is similar to the Mandelbrot-Van Ness representation for the fBm, but in fact it is a compact-interval representation) is proved in \cite{BEG},  in terms of a Wiener  integral (with deterministic integrand):
  \begin{equation}
      B^\mathcal{H}_\alpha (t)=\frac{1}{\sqrt{\Gamma(\alpha)}}\int _{0}^{\infty} \left( \log \frac{t}{s}\right)_{+}^{(\alpha-1)/2}dB(s)=\frac{1}{\sqrt{\Gamma(\alpha)}}\int _{0}^{t} \left( \log \frac{t}{s}\right)^{(\alpha-1)/2}dB(s), \label{mvn}
  \end{equation}
  where $B:=\left\{B(t)\right\}_{t \geq 0}$ is a standard Brownian motion. By considering that $_{\mathcal{H}}\mathcal{M}_{-}^\alpha
1_{[a,b)}\in L_{2}(\mathbb{R}^{+}),$ for $\alpha \in (0,1) \cup (1,2)$ (see Lemma 2.1 in \cite{BEG}), the process $B^\mathcal{H}_\alpha $ is well-defined in this range of values. On the other hand, the integral 
\eqref{mvn} is defined also for all $\alpha\ge 2$ as well, but the latter case (especially the case $\alpha\ge 3$) is  not very interesting for the reasons provided in  Remarks \ref{rem2.1}--\ref{rem2.3}.

Having, in a certain sense, a dual nature, the $\mathcal{H}$-fBm is interesting both from the point of view of functional analysis and as a stochastic process. From the latter point of view, it is a Gaussian-Volterra process of the form $ B^\mathcal{H}_\alpha (t)=\int_0^tK(t,s)dB(s), $ where the kernel is logarithmic, namely, $K(t,s)=K_\alpha\left(\log\frac{t}{s}\right)_+^{(\alpha-1)/2},$ for $ \alpha>0$ and $ K_\alpha>0$. 

Since great attention was previously devoted mainly to power kernels (which arise in the fractional Brownian motion and Riemann-Liouville fractional Brownian motion), it is very interesting to study what happens when the power kernel is replaced by a logarithmic one. 

Moreover, since  $\log t-\log s=\theta^{-1}(t-s),$  where $\theta\in(s,t), $ it is logical to assume that the smoothness properties of the process trajectories will be different near zero and if we deviate from it. 
With this in mind, various  properties of the  process $ B^\mathcal{H}_\alpha$ are considered in Sections \ref{pathwise}--\ref{locnondeterm} and Section \ref{reprod}, while, in Section  \ref{integral}, we construct the Wiener integral w.r.t. $\mathcal{H}$-fBm.

More precisely, after some preliminaries in Section \ref{prelim},  we establish the H\"{o}lder properties of the trajectories 
of $\mathcal{H}$-fBm, in Section \ref{pathwise}. These properties essentially depend on the value of $\alpha $. We establish that for $\alpha\in(0,1)$ the trajectories are H\"{o}lder continuous up to order $\alpha/2$ which is quite predictable. However, for $\alpha>1$ the degree of smoothness of the trajectories is frozen, and they become H\"{o}lder continuous up the order $1/2$. In both cases we consistently prove that the order of H\"{o}lderness cannot be increased, and we also consider in detail the
 effect of  ``increasing smoothness with departure from zero". This characteristic is not observed in the case of fBm, because the latter has stationary increments, and therefore  is, in a sense, homogeneous in time. 
 
 In Section \ref{quasi} it is proved that  Hadamard fractional Brownian motion is a generalized quasi-helix process. It is well-known that fBm is a helix, and, for example, subfractional Brownian motion is a quasi-helix, according to the geometrical terminology of P. Kahane  \cite{Kahane}. 
  The  notion of a generalized quasi-helix was introduced in \cite{BoMiNoZh} and it seemed a bit artificial because Gaussian Volterra processes with power-law kernels are rather quasi-helices. Therefore it was even more interesting to discover that the $\mathcal{H}$-fBm is a natural and non-trivial example of generalized quasi-helix.

In Section \ref{power}, we consider the power variations of $\mathcal{H}$-fBm (defined not as supremums over all partitions but as limits when  the partition diameter tends to zero). We investigate in detail the behavior of this limit for various powers, also observing an interesting effect: despite the fact that $\mathcal{H}$-fBm trajectories appear to have the same smoothness as those of the Wiener process, the quadratic variation of $\mathcal{H}$-fBm is zero. This sharply distinguishes it from the Wiener process. 
  Section \ref{locnondeterm} contains the proof of the local nondeterminism of $\mathcal{H}$-fBm. 

Section \ref{integral} presents the general construction of the Wiener integral with deterministic integrands w.r.t. $\mathcal{H}$-fBm, following a series of preliminary results, which are also interesting on their own: in particular, we prove that the linear span of the set $\left\{_{\mathcal{H}}\mathcal{M}^\alpha_- 1_{(a,b)}; a,b \in \mathbb{R}^+ , a<b\right\}$ is dense in $L_2(\mathbb{R}^+)$. Moreover, by resorting to the H\"{o}lder properties, we also construct this integral as the limit of Riemann sums, for sufficiently smooth integrands. Then, we prove that the representation \eqref{mvn} can be inverted.

 In Section \ref{reprod} we apply ``multiplicative Sonine  pairs'' to the treatment of the properties of the Reproducing Kernel Hilbert Space (RKHS) of $\mathcal{H}$-fBm, and, as a result, establish the law of iterated logarithm. Section \ref{app} contains some auxiliary results.

\section{Preliminaries}\label{prelim}
We recall that the 
\textit{right-sided Hadamard-type integral} is defined, 
as follows:%
\begin{equation}(_{\mathcal{H}}\mathcal{I}_{-, \mu}^{\beta }f)(t) :=\frac{1}{\Gamma (\beta )}
\int_{t}^{\infty }\left(\frac{t}{z}\right)^\mu\left( \log \frac{z}{t}\right) ^{\beta -1}\frac{f(z)}{z}%
dz,  \label{intmin}
\end{equation}%
for $t>0$, $\beta, \mu \in \mathbb{C},$ $\Re(\beta )>0$, where $\Re(\cdot)$ denotes the real part (see \cite{KIL}, equation (2.7.6)).
The 
\textit{right-sided Hadamard-type derivative} is given by%
\begin{equation}
(_{\mathcal{H}}\mathcal{D}_{-, \mu}^{\beta }f)(t):=t^{\mu}\left( -t%
\frac{d}{dt}\right) ^{n}\left[t^{-\mu} (_{\mathcal{H}}\mathcal{I}_{-, \mu}^{n-\beta }f)(t)\right],   \label{dermin}
\end{equation}%
where $\beta \notin \mathbb{N}$, $\Re(\beta)>0$, $n=\left\lfloor \Re(\beta) \right\rfloor +1$ and $t>0$ (see 
(2.7.12) in \cite{KIL}, for $a=0$ and $b=\infty )$. For  $\beta =m \in \mathbb{N}$ 
$(_{\mathcal{H}}\mathcal{D}_{-, \mu}^{\beta }f)(t):=(-1)^mt^\mu \left( t%
\frac{d}{dt}\right) ^{m}(t^{-\mu}f(t))$.
 Hereafter, we will write for brevity, in the case $\mu=0$, 
 $_{\mathcal{H}}\mathcal{I}_{-}^{\beta}:=\thinspace_{\mathcal{H}}\mathcal{I}_{-, 0}^{\beta }$
 and $_{\mathcal{H}}\mathcal{D}_{-}^{\beta }:=\thinspace_{\mathcal{H}}\mathcal{D}_{-, 0}^{\beta }$.

 The operator 
 \eqref{intmin} introduced above is well
defined in the space \[\mathcal{X}_\gamma^p:=\left\{ h: \left(\int_0 ^{\infty} |z^\gamma h(z)|^p \frac{dz}{z}\right)^{1/p} < \infty \right\}, \qquad 
p \in [1, \infty), \gamma \in \mathbb{R},\]
for $\gamma>-\Re(\mu)$.
For $\gamma=1/p$, $\mathcal{X}_{1/p}^p=L_p (\mathbb{R}^+)$ (see \cite{BUT} and \cite{KIL}, p. 113, for more details). Since we will be interested only in the case $\mu=0$, we can always choose, in the following, $\gamma=1/p$ so that we will always consider the operator  in the space  $L_p (\mathbb{R}^+)$ 
    

We will use the following result.
\begin{lem}(\cite{BEG})\label{lem2.1}
Let $_\mathcal{H}\mathcal{D}_{-}^{(1-\alpha)/2 }$ be the right-sided Hadamard derivative
defined in (\ref{dermin}). Then, for $x\in \mathbb{R}^{+}$ and $0 \leq a<b,$%
\begin{equation}
\left( _\mathcal{H}\mathcal{D}_{-}^{(1-\alpha)/2 }1_{[a,b)}\right) (x)=\frac{1}{\Gamma
((\alpha+1)/2 )}\left[ \left( \log \frac{b}{x}\right) _{+}^{(\alpha-1)/2 }-\left(
\log \frac{a}{x}\right) _{+}^{(\alpha-1)/2}\right] ,  \label{had4}
\end{equation}%
where $\left( x\right) _{+}:=x1_{x\geq 0},$ and $_\mathcal{H}\mathcal{D}_{-}^{(1-\alpha)/2
}1_{[a,b)}\in L_{2}(\mathbb{R}^{+}),$ for $\alpha \in (0,1).$ Analogously,
let $_\mathcal{H}\mathcal{I}_{-}^{\alpha/2 }$ be the right-sided Hadamard integral
defined in (\ref{intmin}), then
\begin{equation}
\left( _\mathcal{H}\mathcal{I}_{-}^{(\alpha-1)/2 }1_{[a,b)}\right) (x)=\frac{1}{\Gamma ((\alpha+1)/2)}\left[ \left( \log \frac{b}{x}\right) _{+}
^{(\alpha-1)/2 }-\left( \log \frac{a}{x}\right) _{+} ^{(\alpha-1)/2 }\right] ,  \label{had6}
\end{equation}%
and $_\mathcal{H}\mathcal{I}_{-}^{(\alpha-1)/2 }1_{[a,b)}\in L_{2}(\mathbb{R}^{+}),$ for $%
\alpha \in (1,2).$
\end{lem}

We recall that the Hadamard fractional derivative defined in \eqref{dermin} is the left inverse  operator of the Hadamard fractional integral defined in \eqref{intmin} (for functions in $\mathcal{X}_\gamma^p(\mathbb{R}^+)$ and vice-versa (for functions in $_{\mathcal{H}}\mathcal{I}_{-}^{\beta }(\mathcal{X}_\gamma^p(\mathbb{R}^+))$), as proved in \cite{KIL} (Property 2.28 and Lemma 2.35, respectively). Since $\mu=0$, we have again that $\mathcal{X}_\gamma^p(\mathbb{R}^+)=L_p(\mathbb{R}^+)$, and  we can choose $p=2$.

\begin{lem}\label{leminverse}
Let $\beta \in (0,1)$, then,  for $t\geq 0,$ 
\begin{equation}\label{redint}
\left(_{\mathcal{H}}\mathcal{D}_{-}^{\beta } \; _{\mathcal{H}}\mathcal{I}_{-}^{\beta } 1_{[0,t)}\right)(x)=1_{[0,t)}(x)
\end{equation}
and
\begin{equation}\label{redder}
\left(_{\mathcal{H}}\mathcal{I}_{-}^{\beta } \; _{\mathcal{H}}\mathcal{D}_{-}^{\beta } 1_{[0,t)}\right)(x)=1_{[0,t)}(x)
\end{equation}
    \end{lem}
\begin{proof}
Equation \eqref{redint} can be verified as follows: 
\begin{eqnarray}\label{check}
\left(_{\mathcal{H}}\mathcal{D}_{-}^{\beta } \; _{\mathcal{H}}\mathcal{I}_{-}^{\beta } 1_{[0,t)}(\cdot)\right)(x)&=&\frac{1}{\Gamma(\beta+1)}\left(_{\mathcal{H}}\mathcal{D}_{-}^{\beta } \left(\log\frac{t}{\cdot}\right)_+^\beta\right)(x) \notag \\
&=&-\frac{x}{\Gamma(\beta+1)}\frac{d}{dx}\left[\left(_{\mathcal{H}}\mathcal{I}_{-}^{1-\beta }\left(\log\frac{t}{\cdot}\right)^{\beta}1_{[0,t)}(\cdot)\right)(x)\right] \notag \\
&=&-\frac{x}{\Gamma(\beta+1)\Gamma(1-\beta)}\frac{d}{dx}\int_x^{+\infty} \left(\log\frac{z}{x}\right)^{-\beta}\left(\log\frac{t}{z}\right)^{\beta}1_{[0,t)}(z)\frac{dz}{z} \notag \\
&=&-\frac{x 1_{[0,t)}(x)}{\Gamma(\beta+1)\Gamma(1-\beta)}\frac{d}{dx}\int_0^{\log(t/x)}y^{-\beta} \left(\log \frac{t}{x}-y  \right)^\beta dy \notag \\
&=&\frac{\beta 1_{[0,t)}(x)}{\Gamma(\beta+1)\Gamma(1-\beta)}\int_0^{1}w^{-\beta} \left(1-w  \right)^{\beta-1} dw=1_{[0,t)}(x),  \notag
\end{eqnarray}
which holds only for $\beta>0.$
Formula \eqref{redder} can be checked by similar steps.
\end{proof}

We finally recall that the Hardy–Littlewood theorem holds for the Hadamard fractional integral (
see equation (2.7.30) of Lemma 2.33, with $\mu=0$, in \cite{KIL}).

\begin{lem}\label{lemHL}(\cite{KIL})
Let $\Re(\beta)>0,$  $1 \leq p \leq \infty$ and any $\gamma>0,$ then the operator $_{\mathcal{H}}\mathcal{I}_{-}^{\beta }$ is bounded in $\mathcal{X}_\gamma^p(\mathbb{R}^+)$ as follows
\begin{equation}
\Vert _{\mathcal{H}}\mathcal{I}_{-}^{\beta }f\Vert_{\mathcal{X}_\gamma^p} \leq K \Vert f\Vert_{\mathcal{X}_\gamma^p}, \label{HL}
\end{equation}
where $K:=\gamma^{-\Re(\beta)}$.
\end{lem}

\section{Path-wise properties of the $\mathcal{H}$-fBm}\label{pathwise}
 We start by considering the basic properties of trajectories of $\mathcal{H}$-fBm. 
First, note that its variance can be obtained from \eqref{mvn} as   \begin{equation}\label{varvar}\mathbb{E}(B^\mathcal{H}_\alpha(t))^2=\frac{t}{\Gamma(\alpha)} \int_0^1 \left(\log\frac1z\right)^{\alpha-1}dz=\frac{t}{\Gamma(\alpha)} \int_0^\infty u^{\alpha-1}e^{-u}du=t\end{equation}
and also the $1/2$-self-similarity of $\mathcal{H}$-fBm follows from it.  Moreover, for any $t,s>0$
\begin{equation}\label{varGadamar} \sigma^\alpha_{t,s} =\mathbb{E} (B^\mathcal{H}_\alpha(t) B^\mathcal{H}_\alpha(s)) =\frac{1}{\Gamma(\alpha)}\int_0^{t\wedge s}\left(\log\frac tz\right)^{(\alpha-1)/2}\left(\log\frac sz\right)^{(\alpha-1)/2}dz. \end{equation}
Obviously,  \eqref{varvar} and \eqref{varGadamar} coincide with \eqref{sigma} for $\alpha \in (0,1) \cup (1,2)$. In general, we will consider such values of $\alpha$, but, on some occasions, we will also include other values into consideration, and the range of values  will be discussed separately each time.

\subsection{Smoothness of trajectories: upper bounds, H\"{o}lder properties}
Now, let us study the regularity of the trajectories of   $\mathcal{H}$-fBm $B^\mathcal{H}_\alpha$, using  the   integral representation \eqref{mvn}. In particular, we will use the following definition.
\begin{definition}\label{def1-1}
We say that a process $X:=\left\{X(t)\right\}_{t \geq 0}$ belongs to $\mathcal{C}^{\beta-}([0,T])$, if, for any $\beta' \in (0,\beta)$, there exists a random variable $C(\omega,\beta')$ such that, for $\omega \in \Omega_ {\beta'}$, with $\mathbb{P}(\Omega_ {\beta'})=1$, the following upper bound holds
\[
\vert X(t)-X(s)\vert \leq C(\omega,\beta')(t-s)^{\beta'},
\]
for $0 \leq s <t \leq T$.
\end{definition}
\begin{remark} Since $ \Omega_ {\beta'}\supset \Omega_ {\beta"}$ for $\beta'<\beta"$, we can put  in Definition \ref{def1-1} the same $ \Omega'$ with $\mathbb{P}(\Omega')=1$.
    
\end{remark}
We recall the following well-known results, which provide a sufficient condition for H\"older continuity of a Gaussian process and a uniform modulus of continuity (see Corollary 1, p. 192 in \cite{GIK}  and \cite{LIF}, respectively).

\begin{lem}[\cite{GIK}]\label{gik}
Let $T>0$, $\beta >0$ and let $X:=\left\{X(t)\right\}_{t \geq 0}$ be a Gaussian process such that
\begin{equation}
\mathbb{E}\vert X(t)-X(s)\vert^2 \leq C_T(t-s)^{2\beta}, \label{hol1}
\end{equation}
for $0 \leq s <t \leq T$    and for some  constant $C_T>0$. Then $X$ belongs to 
$\mathcal{C}^{\beta-}([0,T])$.  
\end{lem}

\begin{lem}[\cite{LIF}]\label{lif}
Let $T>0$ and let $X:=\left\{X(t)\right\}_{t \geq 0}$ be a Gaussian process satisfying \eqref{hol1}  for some $\beta >0$. Then the function 
\begin{equation}
\Theta(\varepsilon)=\int_0^\varepsilon \vert \log r \vert^{1/2} r^{\beta -1} dr , \qquad \varepsilon >0, \label{mod1}
\end{equation}
can serve as a uniform modulus of continuity of $X$ on $[0,T]$, w.r.t. $\rho(s,t):= \vert t-s \vert$, in the sense that 
\begin{equation}\limsup_{\varepsilon \to 0^+} \frac{1}{\Theta(\varepsilon)} \sup_{\overset{\vert t-s \vert < \varepsilon}{s,t \in [0,T]}} \vert X(t)-X(s) \vert < \infty, \qquad a.s. \label{mod2}
\end{equation}
\end{lem}
\begin{remark} It is clear that $\beta$ in \eqref{hol1} can be chosen non-uniquely, but the best choice is the higher possible value. In this connection, the reader will see that in Remark \ref{rem2.2} a uniform modulus of continuity for $B^\mathcal{H}_\alpha$ is chosen optimally. 
    
\end{remark}
In the next theorem we provide suitable upper bounds of the form \eqref{hol1}, for the incremental variance of $B^\mathcal{H}_\alpha$ and for proper choices of the power index  $\beta$, while, in the subsequent result, we prove their optimality. We shall apply   the Lagrange theorem for the logarithm and for the power functions, and also  the   inequality $x^\gamma-y^\gamma \leq (x-y)^\gamma$, which holds for $0<y<x$ and $\gamma\in (0,1)$, without mentioning it explicitly.
\begin{theorem}\label{thm1}
Let $T>0$ and $0\leq s\leq t\leq T$. Then the following statements are true.  \begin{enumerate}[label=(\roman*)]
    \item For $\alpha \in (0,1)$ \begin{equation}\label{eq:upperh1}\mathbb{E}\vert B^\mathcal{H}_\alpha(t)-B^\mathcal{H}_\alpha(s)\vert^2\leq C_{T,\alpha}(t-s)^\alpha,
 \end{equation}
 where $C_{T,\alpha} >0$   depends on $T$ and $\alpha$. Therefore $B^\mathcal{H}_\alpha$ belongs
  to $\mathcal{C }^{\alpha/2-}([0,T])$,  for any $T\in (0, \infty)$.  
     \item For $\alpha \in (1,2)$,
       \begin{equation}\label{eq:upperh2}
        \mathbb{E}\vert B^\mathcal{H}_\alpha(t)-B^\mathcal{H}_\alpha(s)\vert^2\leq C_\alpha (t-s),\end{equation}
       where $C_\alpha>0$  depends only on $\alpha $ and does not depend on  $T>0$.
 Therefore  $B^\mathcal{H}_\alpha$ belongs
  to $\mathcal{C }^{1/2-}([0,T])$, for any $T\in (0, \infty)$.
\end{enumerate}
\end{theorem}
\begin{proof} Let $T>0$ be fixed and $0\le s \le t\le T$. Consider the case $s=0$. Then $$ \mathbb{E}\vert B^\mathcal{H}_\alpha(t)-B^\mathcal{H}_\alpha(s)\vert^2=t,$$ whence the necessary upper bounds hold with $T^{1-\alpha}$ and 1, respectively. Therefore, in what follows, we assume that $s>0$.
Applying representation  \eqref{mvn}, we can write that
\
\begin{eqnarray}
 \mathbb{E}\vert B^\mathcal{H}_\alpha(t)-B^\mathcal{H}_\alpha(s)\vert^2  
 &=&\frac{1}{\Gamma(\alpha)}\int _{0}^{s} \left[\left( \log \frac{t}{u}\right)^{(\alpha-1)/2} -\left( \log \frac{s}{u}\right)^{(\alpha-1)/2}\right]^2du\label{hol4}\\&+&\frac{1}{\Gamma(\alpha)}\int _{s}^{t} \left( \log \frac{t}{u}\right)^{\alpha-1}du\notag 
  =:  \frac{1}{\Gamma(\alpha)}[J^{(1)}_{s,t}+J^{(2)}_{s,t}]. \notag
\end{eqnarray}
 In the case when $\alpha \in (0,1)$, it is reasonable to rewrite $ J^{(1)}_{s,t} $ in the form 
 \begin{eqnarray}\label{2.5}
J^{(1)}_{s,t}=   s\int _{0}^{+\infty}\frac{e^{-z}\left[\left( \log \frac{t}{s}+z\right)^{(1-\alpha)/2} -z^{(1-\alpha)/2}\right]^2}{z^{1-\alpha}\left(\log \frac{t}{s}+z\right)^{1-\alpha}}dz=:I^{(1)}_{s,t}+I^{(2)}_{s,t},  \notag \\
\end{eqnarray}
where, in the right-hand side, we denote  $I^{(1)}_{s,t}:=s\int _{0}^{1}(\cdot)$ and $I^{(2)}_{s,t}:=s\int _{1}^{+\infty}(\cdot)$.

Now let us  distinguish two cases.
\begin{enumerate}[label=(\roman*)]
    \item Let $\alpha \in (0,1)$. Then, since $e^{-z}\le 1,$ if $ z>0$, by the change of variable $\log \frac{t}{s}/z=u$, we get that 
    \begin{equation}
    I^{(1)}_{s,t} \leq s\int _{0}^{1}\frac{\left[\left( \log \frac{t}{s}+z\right)^{(1-\alpha)/2} -z^{(1-\alpha)/2}\right]^2}{z^{1-\alpha}\left(\log \frac{t}{s}+z\right)^{1-\alpha}}dz=s\left( \log \frac{t}{s}\right)^\alpha \int_{\log(t/s)}^\infty \frac{\left[(u+1)^{(1-\alpha)/2}-1 \right]^2 }{u^{1+\alpha}(u+1)^{1-\alpha}}du. \label{hol3}
    \end{equation}
Note that, at infinity, the integrand in the latter integral behaves as $u^{-\alpha-1}$. Moreover, at zero, the behavior of the integrand can be described as follows: consider the following expansion (around $u=0$), for $\eta \in (0,u)$,
\[
(u+1)^{(1-\alpha)/2}-1=\frac{1-\alpha}{2}u- \frac{1-\alpha^2}{8}(1+\eta)^{-(3+\alpha)/2}u^2 \leq \frac{1-\alpha}{2}u+ \frac{1-\alpha^2}{8}u^2 \leq (1-\alpha)u, 
\]
if $u< 4/(1+\alpha)$. Thus, in the neighborhood of zero, the integrand in \eqref{hol3} can be bounded as
\begin{equation}\label{upbound}
\frac{\left[(u+1)^{(1-\alpha)/2}-1 \right]^2 }{u^{1+\alpha}(u+1)^{1-\alpha}} \leq \frac{(1-\alpha)^2 u^{1-\alpha}}{(u+1)^{1-\alpha}} \leq (1-\alpha)^2,
\end{equation}
so that the integral   \begin{equation}\label{firstI}
I_\alpha:=\int_{0}^\infty \frac{\left[(u+1)^{(1-\alpha)/2}-1 \right]^2 }{u^{1+\alpha}(u+1)^{1-\alpha}}du\end{equation} is correctly defined. Therefore, we can write that
\begin{equation}\label{needlater} I^{(1)}_{s,t}  \leq I_\alpha s \left( \log \frac{t}{s}\right)^\alpha \leq \frac{I_\alpha s}{\theta^\alpha}(t-s)^\alpha,
\end{equation}
for $\theta \in (s,t)$, where, in the last step, we have applied the Lagrange theorem. Noticing that $\frac{s}{\theta^\alpha} \leq s^{1-\alpha} \leq T^{1-\alpha}$, we  obtain the following inequality:
        \begin{equation}
         I^{(1)}_{s,t} \leq I_\alpha T^{1-\alpha}(t-s)^\alpha. \label{hol2}   
        \end{equation}

Now we can bound $I^{(2)}_{s,t}$ in \eqref{2.5}, recalling  that $\log (t/s)+z >z$, if $t>s$ (for the denominator), as follows
\begin{eqnarray}\label{first}
I^{(2)}_{s,t} &\leq& s\int _{1}^{+\infty}z^{2 \alpha-2}e^{-z}\left[\left( \log \frac{t}{s}+z\right)^{(1-\alpha)/2} -z^{(1-\alpha)/2}\right]^2dz \notag \\
&\leq& s \frac{(1-\alpha)^2}{4} \left(\log \frac{t}{s}\right)^2\int _{1}^{+\infty}z^{\alpha-3}e^{-z}dz=:J_\alpha \frac{(1-\alpha)^2}{4} s  \left( \log \frac{t}{s}\right)^2,
\end{eqnarray}
where $J_\alpha=\int _{1}^{+\infty}z^{\alpha-3}e^{-z}dz$, and for the second inequality we used the upper bound 
\[
\left( \log \frac{t}{s}+z\right)^{(1-\alpha)/2} -z^{(1-\alpha)/2}=\frac{1-\alpha}{2}\theta^{-(1+\alpha)/2}\log \frac{t}{s} \leq \frac{1-\alpha}{2}z^{-(1+\alpha)/2}\log \frac{t}{s},
\]
for $\theta \in \left(z,z+\log (t/s)\right)$. Now we can apply to the right-hand side of \eqref{first} the inequality \eqref{auxil-e-ineq}, with $p=2$ and $L_2=1$,   and we conclude that $$I^{(2)}_{s,t}\le  J_\alpha \frac{(1-\alpha)^2}{4}  (t-s) \leq J_\alpha \frac{(1-\alpha)^2}{4}(t-s)^\alpha T^{1-\alpha}.$$

For the last term in \eqref{hol4}, on the one hand,  we have   that  for $s<t/2$ (in this case $t<2(t-s)$), 
\begin{equation}J^{(2)}_{s,t}=t
\int _{0}^{\log(t/s)}e^{-z} z^{\alpha-1}dz \leq t
\int _{0}^{\infty}e^{-z} z^{\alpha-1}dz \leq 2\Gamma(\alpha) (t-s)\leq 2\Gamma(\alpha)T^{1-\alpha} (t-s)^\alpha . \label{hol5}
\end{equation}
 On the other hand, for $s \geq t/2$ and $\theta \in (s,t),$ we have that
 \begin{equation}\begin{gathered}
 J^{(2)}_{s,t}\leq t
\int _{0}^{\log(t/s)} z^{\alpha-1}dz =\frac{t}{\alpha}(\log t- \log s)^\alpha \leq \frac{t}{\alpha \theta^\alpha}(t-s)^\alpha\leq \frac{t}{\alpha (t/2)^\alpha}(t-s)^\alpha\\ \leq \frac{2^\alpha T^{1-\alpha}}{\alpha}(t-s)^\alpha. \label{hol7}
 \end{gathered}\end{equation}
 Combining the bounds from  \eqref{hol4}-\eqref{hol7}, we receive  that
\begin{equation}\label{const}\mathbb{E}\vert B^\mathcal{H}_\alpha(t)-B^\mathcal{H}_\alpha(s)\vert^2\leq C_{T,\alpha}(t-s)^\alpha,\end{equation}
where
\begin{equation*}
  C_{T,\alpha}= (\Gamma(\alpha))^{-1}\max\left\{I_\alpha T^{1-\alpha},   {J_\alpha \frac{(1-\alpha)^2}{4}T^{1-\alpha}},  2\Gamma(\alpha)T^{1-\alpha}, \frac{2^\alpha T^{1-\alpha}}{\alpha}\right\}.\end{equation*} Now,   $(i) $ follows from  Lemma \ref{gik}.
 \item Let $\alpha \in (1,2)$. Then changing the variable $z=\log \frac{t}{s}/u$,  we can rewrite $J^{(1)}_{s,t}$   as
 \begin{eqnarray}
J^{(1)}_{s,t}&=&s\int _{0}^{+\infty}e^{-z}\left[\left( \log \frac{t}{s}+z\right)^{(\alpha-1)/2} -z^{(\alpha-1)/2}\right]^2dz \label{qv3} \\
&=&s\left(\log \frac{t}{s}  \right)^\alpha \int _{0}^{+\infty}u^{-\alpha-1} (t/s)^{-1/u}\left[(u+1)^{(\alpha-1)/2}-1\right]^2du \notag \\
&\leq&s\left(\log \frac{t}{s}  \right)^\alpha \int _{0}^{+\infty}u^{-\alpha-1}\left[(u+1)^{(\alpha-1)/2}-1\right]^2du =:K_\alpha s\left(\log \frac{t}{s}  \right)^\alpha. \notag
 \end{eqnarray}
 Indeed, the last integral converges since the integrand behaves as $u^{1-\alpha}$  at zero, and as $u^{-2}$ at infinity. Now we can apply inequality \eqref{auxil-e-ineq} with $p=\alpha\in(1,2)$, and conclude that  
$$ s\left(\log \frac{t}{s}  \right)^\alpha\le L_\alpha(t-s).$$ Therefore we have that $ J^{(1)}_{s,t} \leq K_\alpha   L_\alpha (t-s) .$

The last term in \eqref{hol4} can be bounded for $s<t/2,$ as in the middle part of \eqref{hol5}, so that we get $J^{(2)}_{s,t}\leq 2\Gamma(\alpha)(t-s)$.

 If $s \geq t/2$,  we can start with inequality \begin{equation}\label{dji2} 
      J^{(2)}_{s,t}\leq t
\int _{0}^{\log(t/s)} z^{\alpha-1}dz =\frac{t}{\alpha}\left(\log\frac{t}{s}\right)^\alpha\le \frac{2s}{\alpha}\left(\log\frac{t}{s}\right)^\alpha  \end{equation} as in  \eqref {hol7}, apply \eqref{auxil-e-ineq} in Appendix, with $p=\alpha $,  and    conclude that $J^{(2)}_{s,t}\leq \frac{2L_\alpha}{\alpha} (t-s).$

Combining  together the bounds obtained so far, we get  that
 \[\mathbb{E}\vert B^\mathcal{H}_\alpha(t)-B^\mathcal{H}_\alpha(s)\vert^2\leq C_\alpha(t-s),
 \]
 with $$C_\alpha= (\Gamma(\alpha))^{-1} \max\left\{ K_\alpha L_\alpha,   2\Gamma(\alpha),\frac{2L_\alpha}{\alpha}\right\}$$ so that $(ii) $ follows, by applying again Lemma \ref{gik}.

\end{enumerate}

\end{proof}
\subsection{Smoothness of trajectories: lower  bounds}
In the next result we prove optimality of  the upper bounds given in Theorem \ref{thm1}. We shall denote by $C$ various constants whose value is irrelevant. 

\begin{theorem}\label{thm2}
    The  upper bounds in \eqref{eq:upperh1} and \eqref{eq:upperh2}  can not be replaced by any other with higher index.
\end{theorem}
\begin{proof}
\begin{enumerate}[label=(\roman*)]
    \item Let $\alpha \in (0,1)$ and let $s=t(1-1/n),$ for $n \geq 1$, then we can write that
    \begin{eqnarray}
    J^{(1)}_{s,t}  &\geq& s\left( \log \frac{t}{s}\right)^\alpha e^{-1}\int_{\log(t/s)}^\infty \frac{\left[(u+1)^{(1-\alpha)/2}-1 \right]^2 }{u^{1+\alpha}(u+1)^{1-\alpha}}du \label{hol8} \\
    &=& t \left(1-\frac{1}{n}\right)\left( \log \left(1+\frac{1}{n-1}\right)\right)^\alpha e^{-1}\int_{\log(1+1/(n-1))}^\infty \frac{\left[(u+1)^{(1-\alpha)/2}-1 \right]^2 }{u^{1+\alpha}(u+1)^{1-\alpha}}du\notag \\
    &\sim& C t \left(\frac{1}{n-1}\right)^\alpha \sim C t \left(\frac{1}{n}\right)^\alpha = C t^{1-\alpha} (t-s)^\alpha, \notag
    \end{eqnarray}
 as $n \to \infty$.  

  \item Let $\alpha \in (1,2)$ and take $s=1/n$ and $t=2/n$, then we have that
  \begin{eqnarray}
   J^{(1)}_{s,t} &=&    s\left(\log \frac{t}{s}  \right)^\alpha \int _{0}^{+\infty}u^{-\alpha-1}(t/s)^{-1/u}\left[(u+1)^{(\alpha-1)/2}-1\right]^2du \notag \\
   &= &\frac{1}{n}\left(\log 2\right)^\alpha \int _{0}^{+\infty}u^{-\alpha-1}2^{-1/u}\left[(u+1)^{(\alpha-1)/2}-1\right]^2du \sim \frac{C}{n}  =C(t-s),
  \end{eqnarray}
  as $n \to +\infty$.

    \end{enumerate}
\end{proof}
\begin{rem}\label{rem2.1} Let us investigate what happens if $\alpha\in[2,3)$. In this case $\frac{\alpha-1}{2}\in [1/2,1)$, and so it is possible to  apply inequality \eqref{auxil-e-ineq}  with $p=\alpha-1\in[1,2)$, and to write down
\begin{align*}J^{(1)}_{s,t} &=\lim_{\varepsilon\downarrow 0}\int _{\varepsilon}^{s} \left[\left( \log \frac{t}{u}\right)^{(\alpha-1)/2} -\left( \log \frac{s}{u}\right)^{(\alpha-1)/2}\right]^2du\\
&\leq\lim_{\varepsilon\downarrow 0} \int _{\varepsilon}^{s} \left[\left( \log \frac{t}{u}\right) -\left( \log \frac{s}{u}\right)\right]^{\alpha-1}du\\
&=s \left( \log \frac{t}{s}\right)^{\alpha-1} \leq L_{\alpha-1}(t-s),
\end{align*}
and the upper bounds for $J^{(2)}_{s,t}$  are preserved as in the proof of item (ii) in Theorem \ref{thm1}. Therefore, statement (ii) of  Theorem \ref{thm1} holds for all $\alpha\in(1,3)$. Moreover, we can put, as in Theorem \ref{thm2}, $s=1/n,\, t=2/n$ and get  from \eqref{hol5} that $$J^{(2)}_{1/n,2/n}=\frac2n
\int _{0}^{\log2}e^{-z} z^{\alpha-1}dz \sim \frac{C}{n}  =C(t-s),$$
   as $n \to +\infty$. Therefore both Theorems \ref{thm1} and \ref{thm2} are valid for $\alpha\in(1,3)$. 
\end{rem}

\begin{rem}\label{remrem} It is interesting to compare the results of Theorems \ref{thm1} and \ref{thm2}, that are obtained for the intervals $[0,T]$ starting from zero, with the situation on the intervals $[\delta, T]$ for $0<\delta<T\le\infty. $ For $\alpha\in(0,1)$ the obtained upper bounds are optimal even if we deviate from zero. However, if $\alpha\in(1,2)$,  it immediately follows from  \eqref{qv3} combined with   \eqref{dji2} (the latter one can be now applied to all $0<s<t\le T$), that  \[\mathbb{E}\vert B^\mathcal{H}_\alpha(t)-B^\mathcal{H}_\alpha(s)\vert^2\leq C_\alpha(t-s)^\alpha,
 \] and therefore $B^\mathcal{H}_\alpha$ belongs
  to $\mathcal{C }^{\frac{\alpha}{2}-}([\delta,T])$, for any $0<\delta<T $, where   H\"{o}lder exponent $\frac{\alpha  }{2}$ is bigger  than $\frac12.$  Let    $\alpha\in(2,3)$. Then, according to  \eqref{qv3}, $$J^{(1)}_{s,t}=s\int _{0}^{+\infty}e^{-z}\left[\left( \log \frac{t}{s}+z\right)^{(\alpha-1)/2} -z^{(\alpha-1)/2}\right]^2dz,$$ and the function $r(z)=\left( \log \frac{t}{s}+z\right)^{(\alpha-1)/2} -z^{(\alpha-1)/2}$ decreases in $z\ge 0$. Therefore 
 $$J^{(1)}_{s,t}\le s \left( \log \frac{t}{s}\right)^{ \alpha-1 }\le \delta^{1-\alpha}(t-s)^{\alpha-1}.$$  This means that   $B^\mathcal{H}_\alpha$ belongs
  to $\mathcal{C }^{\frac{\alpha -1}{2}-}([\delta,T])$, for any $0<\delta<T\le\infty, $ and H\"{o}lder exponent $\frac{\alpha -1}{2}$ is bigger  than $\frac12.$ Thus, we observe an interesting effect of ``increasing smoothness with departure from zero '',  which is not inherent, for example, to fractional Brownian motion, but is observed   for some Gaussian-Volterra processes (see, e.g., the related upper bounds in \cite{MiShSh}).
\end{rem}
\begin{rem}\label{rem2.3} We analyze now the situation where $\alpha\ge 3$. In this case, on any  interval $[0,T]$,  according to \eqref{qv3}, and with the help of Lagrange theorem, Lemma \ref{implem} and inequality $(a+b)^p\le 2^p(a^p+b^p), a,b, p>0$  we get that\begin{eqnarray}\label{bigalpha}
J^{(1)}_{s,t}&=&s\int _{0}^{+\infty}e^{-z}\left[\left( \log \frac{t}{s}+z\right)^{(\alpha-1)/2} -z^{(\alpha-1)/2}\right]^2dz \notag  \\
&\le& \frac{(\alpha-1)^2}{4}s\int _{0}^{+\infty}e^{-z}\left( \log \frac{t}{s}+z\right)^{\alpha-3 }\left( \log \frac{t}{s}\right)^{2}dz\notag  \\
&\le& 2^{\alpha-5}(\alpha-1)^2\left(s\left( \log \frac{t}{s}\right)^{\alpha-1 }+\Gamma(\alpha-2)s\left( \log \frac{t}{s}\right)^{2 }\right)\le C(t-s),
\end{eqnarray}
and the upper bound for $J_{s,t}^2$ follows as in Theorem \ref{thm1}. The lower bound can be obtained, from the first equality in \eqref{bigalpha}, as in Theorem \ref{thm2}; whence we get that the process $B^\mathcal{H}_\alpha$ is H\"{o}lder up to order $1/2$ on any interval $[0,T]$. Again, if we consider, as in Remark \ref{remrem}, any interval $[\delta, T]$ separated from zero, then $$s\left( \log \frac{t}{s}\right)^{p}\le C(t-s)^p,\, p\ge 1,$$ and the inequality \eqref{bigalpha} can be transformed into $$J^{(1)}_{s,t}\le C(t-s)^2,$$ by taking into account that $\alpha-1\ge 2$. Since $$J^{(2)}_{s,t}=\frac{1}{\Gamma(\alpha)}\int _{s}^{t} \left( \log \frac{t}{u}\right)^{\alpha-1}du\le \frac{1}{\Gamma(\alpha)}(t-s)\left( \log \frac{t}{s}\right)^{\alpha-1}\le C(t-s)^{\alpha},$$ we can say that $B^\mathcal{H}_\alpha\in \mathcal{C}^{1-}([\delta,T])$, for any $\alpha\ge 3$ and $\delta>0$. So, the effect of  ``increasing smoothness with departure from zero'' is appearing again, but, in this case, the trajectories are ``almost Lipschitz'' at any $([\delta,T])$. In this sense, such smoothness  characterizes the process $B^\mathcal{H}_\alpha$ with $\alpha\ge 3$ as not very interesting from the point of view of applications. 
\end{rem}
\begin{rem}\label{rem2.2} Now, a few words  about the modulus of continuity: in order to apply Lemma \ref{lif}, we rewrite the modulus of continuity in \eqref{mod1} as follows
    \begin{equation}\notag
    \Theta(\varepsilon)=\int_0^\varepsilon \vert \log r \vert^{1/2} r^{\beta -1} dr=\beta^{-3/2}\int_{-\log \varepsilon^\beta}^\infty x^{1/2} e^{-x} dx \sim\sqrt{\log(1/\varepsilon)}e^{\log \varepsilon^\beta} = \varepsilon^\beta \sqrt{\log(1/\varepsilon)},
    \end{equation}
    where the approximation holds as $\varepsilon\downarrow 0$. 
    Thus, it immediately follows from Theorems \ref{thm1} and \ref{thm2} that a uniform modulus of continuity of $B^\mathcal{H}_\alpha$ on $[0,T]$ is given, for $\alpha \in (0,1)$, by $\Theta(\varepsilon) \sim \varepsilon^{\alpha/2} \sqrt{\log(1/\varepsilon)}$, whence, for $\alpha \in (1,2)$, it is $\Theta(\varepsilon) \sim \varepsilon^{\frac{1}{2}-\rho} \sqrt{\log(1/\varepsilon)}$, for any $0<\rho<1/2.$
\end{rem}
 \subsection{Memory properties of the $\mathcal{H}$-fBm} There can be different approaches to the definition of short and long memory of the process with non-stationary increments. One of these definitions was considered in \cite{BEG}, where it was proved that for $\alpha\in(1,2)$ $\mathcal{H}$-fBm has a long memory while for $\alpha\in(0,1)$ $\mathcal{H}$-fBm is anti persistent, and anti persistence and short memory exclude each other. However, let us consider here the definition of short/long memory in the same terms that are used for fBm. For example, let us  prove that the $\mathcal{H}$-fBm is short-range dependent for $\alpha \in (0,1)$:
to this aim, it is sufficient to check that the series $\sum_{n=1}^\infty \left\vert \mathbb{E}B_\alpha^\mathcal{H}(1)(B_\alpha^\mathcal{H}(n)-B_\alpha^\mathcal{H}(n-1))\right\vert=-\sum_{n=1}^\infty   \mathbb{E}B_\alpha^\mathcal{H}(1)(B_\alpha^\mathcal{H}(n)-B_\alpha^\mathcal{H}(n-1)) $ converges. Indeed, to check   the sign and to bound the value, note that
\begin{eqnarray}\left\vert \mathbb{E}B_\alpha^\mathcal{H}(1)(B_\alpha^\mathcal{H}(n)-B_\alpha^\mathcal{H}(n-1))\right\vert &=&  \int_0^1 \left(\log\frac1z\right)^{(\alpha-1)/2} \left[ \left(\log\frac{n-1}{z}\right)^{(\alpha-1)/2}-\left(\log\frac{n}{z}\right)^{(\alpha-1)/2}\right]dz \notag \\
&=&  \int_1  ^\infty\frac{(\log u)^{(\alpha-1)/2}}{u^2} \left[ \left(\log (n-1)u\right)^{(\alpha-1)/2}-\left(\log nu\right)^{(\alpha-1)/2}\right]du. \notag
\end{eqnarray}
Consider  $\rho_n:=\left(\log (n-1)u\right)^{(\alpha-1)/2}-\left(\log nu\right)^{(\alpha-1)/2}=\frac{1-\alpha}{2\Theta}(\log \Theta)^{(\alpha-3)/2}u$, for $\Theta \in [(n-1)u,nu],$ for any $u \geq 1$ and $n \in \mathbb{N}.$ Then for $n \geq 3$,
\begin{equation}
\rho_n\leq \frac{1-\alpha}{2(n-1)}(\log nu)^{(\alpha-3)/2}\leq \frac{1-\alpha}{2}\frac{(\log n)^{(\alpha-3)/2}}{n-1}.
\end{equation}
Moreover, $\int_1^\infty u^{-2}(\log u)^{(\alpha-1)/2}du=\int_0^\infty e^{-z}z^{(\alpha-1)/2}dz <\infty,$ and  $\sum_{n=3}^\infty \frac{1}{n-1}(\log n)^{(\alpha-3)/2}$ converges since $(\alpha-3)/2 <-1$. Very similar calculations are needed to prove that for $\alpha>1$ the series $\sum_{n=1}^\infty \left\vert \mathbb{E}B_\alpha^\mathcal{H}(1)(B_\alpha^\mathcal{H}(n)-B_\alpha^\mathcal{H}(n-1))\right\vert=\sum_{n=1}^\infty   \mathbb{E}B_\alpha^\mathcal{H}(1)(B_\alpha^\mathcal{H}(n)-B_\alpha^\mathcal{H}(n-1)) $ diverges.
 
\subsection{Some plots of trajectories}

We conclude this section with some plots where the trajectories of the $\mathcal{H}$-fBm are compared with those of the Brownian motion, for different values of $\alpha$. They confirm the results obtained so far on their smoothness. Recall that Brownian motion corresponds to $\alpha=1$. Therefore it is natural that on the plot (A)  $\mathcal{H}$-fBm with $\alpha=0.5$ is much more rough,  $\mathcal{H}$-fBm with $\alpha=1.5$ is more smooth but even more smooth when we   retreat from zero. Similar situation is on the plot (B), where we see additionally that, in accordance to Remark \ref{rem2.3}, the trajectories are ``almost Lipschitz'' farther from zero.

\begin{figure}[H]
 \subfloat
 [\centering $\mathcal{H}$-fBm with $\alpha=0.5$ (red), $\alpha=1.5$ (blue) vs Brownian motion (black)]
{{\includegraphics[width=7.5cm]{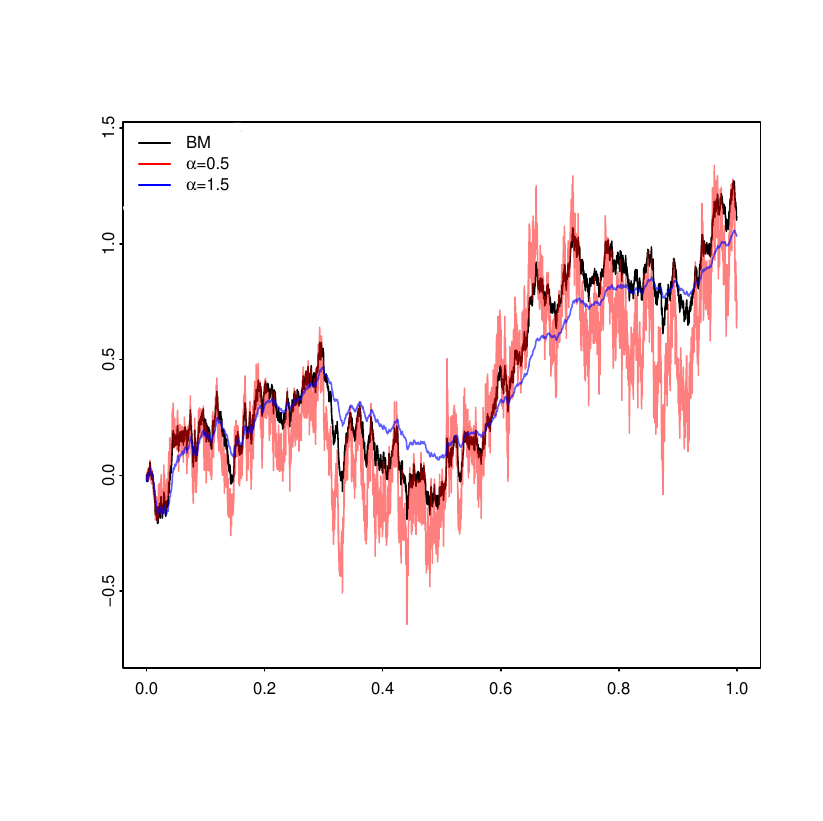}}}
    \qquad
    \subfloat
    [\centering $\mathcal{H}$-fBm with $\alpha=0.5$ (red), $\alpha=1.5$ (blue), $\alpha=3$ (green) vs Brownian motion (black)]
{{\includegraphics[width=7.5cm]
    {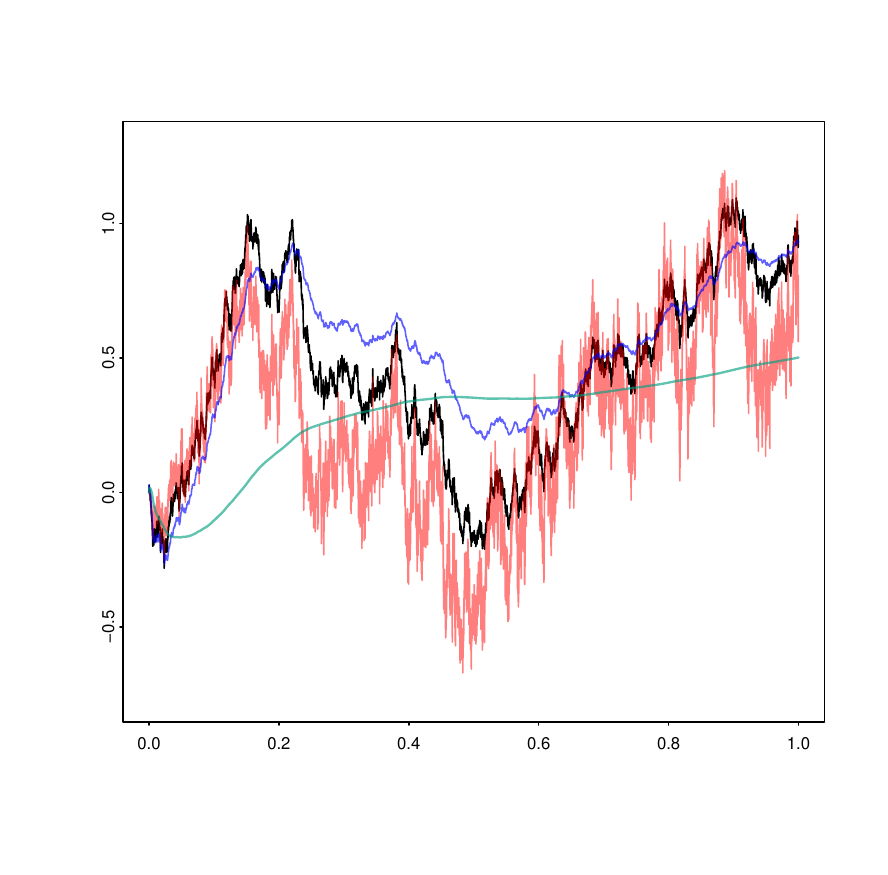}}}
    \label{fig:example}%
    \end{figure}

 \vspace{0.2cm} 

\section{$\mathcal{H}$-fBm as generalized quasi-helix}\label{quasi}
\subsection{Main properties} Let us investigate whether $B^\mathcal{H}_\alpha$ is a quasi-helix or a generalized quasi-helix, according to the classification of \cite{Kahane} and \cite{BoMiNoZh}, respectively. We prove that, due to the irregular behavior  of the incremental variance at zero, our process is not  a quasi-helix; however, it is a generalized quasi-helix. 

\begin{definition}(\cite{Kahane})
The stochastic process $X=\{X_t, t\in[0,T]\}$ is called  $(\gamma,T)$-quasi-helix if there exist two constants $C_1,C_2>0$ and index $\gamma>0$ such that, for any $s,t\in[0,T]$, 
$$C_1|t-s|^\gamma\le \mathbb{E}|X_t-X_s|^2\le C_2|t-s|^\gamma .$$
\end{definition}
\begin{definition}(\cite{BoMiNoZh}) The stochastic process $X=\{X_t, t\in[0,T]\}$ is called  $(\gamma_1,\gamma_2,T)$-generalized quasi-helix if there exist two constants $C_1,C_2>0$ and two (possibly, different) indices $\gamma_1, \gamma_2>0$ such that for any $s,t\in[0,T]$ $$C_1|t-s|^{\gamma_1}\le \mathbb{E}|X_t-X_s|^2\le C_2|t-s|^{\gamma_2} .$$
 \end{definition}   
Note that in this case $\gamma_1\ge \gamma_2.$
\begin{theorem} For any $T>0$, the stochastic process $B^\mathcal{H}_\alpha$ is   $(1, \alpha, T)$- generalized quasi-helix in the case $\alpha\in(0,1)$ and   $(\alpha,1,T)$- generalized quasi-helix in the case $\alpha\in(1,2)$.
    
\end{theorem}

\begin{proof}The respective upper bounds were already  established   in Theorem \ref{thm1}, and the impossibility of the increasing the degrees $\alpha$ and $1$, respectively, was confirmed in Theorem  \ref{thm2}.  Therefore,  our goal is to establish that, for any $T>0$, there exists  $C_1>0$ such that, for any $0\le s\le t\le T$,
$$\mathbb{E}\vert B^\mathcal{H}_\alpha(t)-B^\mathcal{H}_\alpha(s)\vert^2\ge C_1(t-s), $$ in the case  $\alpha\in(0,1)$, while 
 $$\mathbb{E}\vert B^\mathcal{H}_\alpha(t)-B^\mathcal{H}_\alpha(s)\vert^2\ge C_1(t-s)^\alpha,$$ in the case  $\alpha\in(1,2)$, and 
 that in both cases it is impossible to decrease the  powers in the right-hand side. 
 \begin{itemize}
     \item[$(i)$] Let $\alpha\in(0,1)$; it is enough to prove that  there exists a constant $C_1>0$ such that  \begin{equation}\label{firstineqq}J^{(2)}_{s,t}=t
\int _{0}^{\log(t/s)}e^{-z} z^{\alpha-1}dz\ge C_1(t-s).
\end{equation} If $t\ge e s$, then $\log(t/s)\ge 1$, and we can put $C_1=\int _{0}^{1}e^{-z} z^{\alpha-1}dz$. Let instead $t< e s$ and denote $r=\log(t/s)<1$; then \eqref{firstineqq} is equivalent to the inequality $$\int_0^re^{-z} z^{\alpha-1}dz\ge C_1(1-e^{-r}). $$ The latter inequality becomes equality if $r=0$, and the derivatives of the left- and right-hand side equal $e^{-r} r^{\alpha-1}$ and $C_1e^{-r}$, respectively. Therefore here it is enough to put $C_1=1$, and finally $C_1=\min\{1, \int _{0}^{1}e^{-z} z^{\alpha-1}dz\}. $ Furthermore, if $s=0$, then it follows from the formula \eqref{varvar} that the power 1 of $t-s$ in the inequality \eqref{firstineqq} can not be decreased. 
\item[$(ii)$] Let $\alpha\in(1,2)$;
it is enough to prove that  there exists a constant $C_1>0$ such  that  \begin{equation}\label{secondineqq}J^{(2)}_{s,t}=t
\int _{0}^{\log(t/s)}e^{-z} z^{\alpha-1}dz\ge C_1(t-s)^\alpha.
\end{equation}
Again, let  $t\ge e s$ and denote $r=\log(t/s)\ge 1.$ Then \eqref{secondineqq} is equivalent to $$\int _{0}^{r}e^{-z} z^{\alpha-1}dz\ge C_1(1-e^{-r})^\alpha t^{\alpha-1}.$$ Obviously, this inequality holds with $C_1=T^{1-\alpha}\int _{0}^{1}e^{-z} z^{\alpha-1}dz.$ Now, let $r<1$. Since $1-e^{-r}\le r$, it is enough to choose $C_1$ in such a way that the inequality $$\int _{0}^{r}e^{-z} z^{\alpha-1}dz\ge C_1 r^\alpha T^{\alpha-1} $$ holds. The latter inequality becomes equality if $r=0$, and the derivatives of the left- and right-hand side equal $e^{-r} r^{\alpha-1}$ and $  C_1 T^{\alpha -1}\alpha r^{\alpha-1}$, respectively. Therefore, it is sufficient to choose $C_1=\frac{T^{1-\alpha}}{\alpha e}$. Finally, inequality \eqref{secondineqq} holds for $C_1=T^{1-\alpha}\min\{\int _{0}^{1}e^{-z} z^{\alpha-1}dz, \frac{1}{\alpha e}\}$. The power $\alpha$ can not be decreased because for any $\beta<\alpha$
$$\lim_{r\to 0}\frac{\int _{0}^{r}e^{-z} z^{\alpha-1}dz}{(1-e^{-r})^\beta}=\lim_{r\to 0}\frac{ e^{-r} r^{\alpha-1}}{\beta r^{\beta-1}}=0. $$
 \end{itemize}
\end{proof}
\subsection{Behavior of $\mathcal{H}$-fBm in the ``boundary'' values of parameter $\alpha$} Let us  study the behavior of the process  $B^\mathcal{H}_\alpha$ as  function of the parameter $\alpha$, at the boundary points
$\alpha= 0, 1,2$. According to equality  \eqref{varvar}, for any $t\ge 0$, $\mathbb{E}(B^\mathcal{H}_\alpha(t))^2=t$ and does not depend on $\alpha$.  Note that the point $\alpha=2$ is not somehow critical or boundary for the process defined by \eqref{mvn} because this process  exists for any $\alpha>1 $ and, investigating the behavior of trajectories,  we did not observe something crucial at point $\alpha=2$. Therefore,  let us consider the behavior of trajectories of $B^\mathcal{H}_\alpha$ in the neighborhood of points $\alpha= 0, 1$. To compare, for fBm ``boundary'' points are $H=0$, where fBm transforms into the specific   white noise  $\frac{\xi_t-\xi_0}{\sqrt{2}}$, where $\xi_t, t\ge 0$ is  the family of independent  standard Gaussian  variables  (\cite{BoMiNoZh}), $H=1/2$, where it transforms into standard Brownian motion, and $H=1$, where it transforms into the linear process $\xi t$, where $\xi$ is a standard Gaussian random variable. First, Let us establish what is $\mathcal{H}$-fBm at zero. 
 \begin{theorem}\label{proposi3.1} As $\alpha\to 0$, finite-dimensional distributions of $B^\mathcal{H}_\alpha$ weakly converge to the finite dimensional distributions of the Gaussian normalized white noise $$B^\mathcal{H}_0(t)=t^{1/2}\xi_t,$$
 where $\xi_t$ is a standard white noise. 
  \end{theorem}
  \begin{proof} It is sufficient to study the asymptotic behavior of covariance \begin{equation*}  \sigma^\alpha_{t,s} =\frac{1}{\Gamma(\alpha)}\int_0^{t\wedge s}\left(\log\frac tz\right)^{(\alpha-1)/2}\left(\log\frac sz\right)^{(\alpha-1)/2}dz. \end{equation*} Let $0<s<t$. Then we change the variables $\log\frac sz=u$ and get that 
  \begin{equation*}  \sigma^\alpha_{t,s} =\frac{s}{\Gamma(\alpha)}\int_0^{\infty}u^{(\alpha-1)/2}\left(u+\log\frac ts\right)^{(\alpha-1)/2}e^{-u}du. \end{equation*}
  If $\alpha\to 0$, then $\Gamma(\alpha)\to\infty$ while the Lebesgue dominated theorem supplies that $$\int_0^{\infty}u^{(\alpha-1)/2}\left(u+\log\frac ts\right)^{(\alpha-1)/2}e^{-u}du\to \int_0^{\infty}u^{  -1/2}\left(u+\log\frac ts\right)^{ -1/2}e^{-u}du<\infty,$$ whence $\sigma^\alpha_{t,s}\to 0.$ Furthermore, $\sigma^\alpha_{t,t} =t$, and the proof immediately follows.
  \end{proof}   
It follows from Theorem \ref{proposi3.1} that  the  process $\alpha^\varepsilon B^\mathcal{H}_\alpha$ reduces to an identically zero process as $\alpha\to 0$, for any $\varepsilon>0$. Now let us establish which normalizing factor suppresses the maximum.
 \begin{theorem} For any $T>0$ and any $\epsilon>0$ 
\begin{equation}\label{epsilon}
    \lim_{\alpha\to 0} \alpha^{1/2+\epsilon} \mathbb{E}\max_{0\le t\le T}|B^\mathcal{H}_\alpha(t)|=0.\end{equation}
    
\end{theorem}
\begin{proof} Due to the self-similarity of $B^\mathcal{H}_\alpha$, it is enough to consider only the interval $[0,1]$. We shall apply  Theorem 2.1 given in \cite{BoMiNoZh} . According to the second statement of the latter, if there exist $C>0$  and $H\in(0,1)$ such that for a Gaussian process $X$ $$\mathbb{E}|X_t-X_s|^2\le C|t-s|^{2H}, \qquad t,s\in[0,1],$$ then $$\mathbb{E}\max_{0\le t\le 1} X_t< \frac{16.3\sqrt{C}}{\sqrt{H}}. $$ In our case $2H=\alpha$ and $C=\alpha^{1+2\epsilon}  C_{1, \alpha}$, where $C_{T, \alpha}$ is presented  in \eqref{const}.
 Taking into account the symmetry of $B^\mathcal{H}_\alpha$ and ignoring constant multipliers, we see that it is enough to prove that \begin{equation}\label{limlim}\lim_{\alpha\to 0} \alpha^{2\epsilon}C_{1, \alpha}=0.
     \end{equation}  Let us analyze the behavior of the components whose maximum is taken in \eqref{const}, as $\alpha\to 0$. First, consider $ (\Gamma(\alpha))^{-1} I_\alpha  $, taking into account \eqref{upbound} and \eqref{firstI}. We see from    \eqref{upbound} that, for  $\alpha<1$, the integrand in $I_\alpha$ is bounded, at least for $u<2$, while, for $u>2$, the integrand is bounded by $u^{-1-\alpha}$. So it is enough to prove that $ \alpha^{2\epsilon}(\Gamma(\alpha))^{-1} \int_2^\infty u^{-1-\alpha}du\to 0,$ as $\alpha\to 0$. Note that $\alpha\Gamma(\alpha)=\Gamma(\alpha+1)\to 1$, as $\alpha\to 0$, and $\int_2^\infty u^{-1-\alpha}du=2^{-\alpha}/\alpha$. Therefore, as $\alpha\to 0$,  
$$ \alpha^{2\epsilon}(\Gamma(\alpha))^{-1} \int_2^\infty u^{-1-\alpha}du=\frac{\alpha^{ 2\epsilon}2^{-\alpha}}{ \Gamma(\alpha+1)}\to 0.$$ The other components in \eqref{const} can be considered similarly, but even more simply, because, for example,  integral $J_\alpha$ is bounded.  So, $C_{1,\alpha}$ is bounded, and \eqref{limlim} easily follows and theorem  is proved. 
\end{proof}
\begin{remark} Applying Fernique's  upper and lower bounds for Gaussian processes (see, e.g., \cite{lishao}), we can conclude that, for any  $T>0$, $$C_1^{1/2} \mathbb{E}\max_{0\le t\le T}B(t)\le \mathbb{E}\max_{0\le t\le T} B^\mathcal{H}_\alpha(t) \le (C_{T,\alpha})^{1/2}\mathbb{E}\max_{0\le t\le T} B^{\alpha/2}(t),$$ where $B$ is a Wiener process, and $B^{\alpha/2}$ is a fractional Brownian motion with Hurst index $\alpha/2.$  It means that $\mathbb{E}\max_{0\le t\le T} B^\mathcal{H}_\alpha(t)$ does not tend to zero as $\alpha\to 0$. However, at the moment, we can not essentially improve the multiplier $\alpha^{1/2+\epsilon}$ in \eqref{epsilon}.
    
\end{remark}
For $\alpha=1$ the $\mathcal{H}$-fBm reduces to a Wiener process. Now  we just establish that, for any $t>0$, $\mathcal{H}$-fBm approximates  the underlying  Wiener process in $L_2(\Omega)$, as $\alpha\to 1$, and we estimate  the rate of convergence. 
\begin{theorem} There exist    $\rho>0$ and $C>0$ such that for any $t>0$ and $\alpha\in(1-\rho,1+\rho)$
$$\mathbb{E}(B^\mathcal{H}_\alpha(t)-B(t))^2\le Ct(\alpha-1)^2.$$
    
\end{theorem}
\begin{proof} It is sufficient to consider $t=1$ because both processes are self-similar with index $1/2$. So, we bound the value 
\begin{equation}\begin{gathered}\label{gammasecond}
 \mathbb{E}(B^\mathcal{H}_\alpha(1)-B(1))^2=(\Gamma(\alpha))^{-1}\int_0^1\left[\left(\log\frac1s\right)^{\frac{\alpha-1}{2}}-\sqrt{\Gamma(\alpha)}\right]^2ds\\ \le 2(\Gamma(\alpha))^{-1}\int_0^1\left[\left(\log\frac1s\right)^{\frac{\alpha-1}{2}}-1\right]^2ds+2(\Gamma(\alpha))^{-1}(\sqrt{\Gamma(\alpha)}-1)^2.
     \end{gathered}   
\end{equation}
It is well-known that the Gamma function admits a unique point of minimum among positive arguments, that this minimum is achieved at point $x=1,4616321\ldots$ and it equals $G_m=0,885603$. Furthermore, $\Gamma'(1)=-\gamma$, where $\gamma=\lim_{n\to\infty}\left(\sum_{k=1}^n \frac1k-\log n\right)$ is the Euler constant. Therefore there exists an interval $(1-\rho_1,1+\rho_1)$, in which $|\Gamma'(x)|<2\gamma$. Then, for   $\alpha\in(1-\rho,1+\rho)$, according to Lagrange theorem and taking into account that $1=\Gamma(1),$ we get $$2(\Gamma(\alpha))^{-1}(\sqrt{\Gamma(\alpha)}-1)^2\le 2(\Gamma(\alpha))^{-1}(\sqrt{\Gamma(\alpha)}+1)^{-2} (\Gamma(\alpha)-1)^{2} \le 2(G_m)^{-2}4\gamma^2(\alpha-1)^2.$$ Furthermore, \begin{equation*}\begin{gathered}\int_0^1\left[\left(\log\frac1s\right)^{\frac{\alpha-1}{2}}-1\right]^2ds=\int_0^\infty(z^ {\frac{\alpha-1}{2}}-1)^2e^{-z}dz=\Gamma(\alpha)-2\Gamma\left(\frac{\alpha+1}{2}\right)+1\\=\int_{\frac{\alpha+1}{2}}^\alpha\Gamma'(x)dx-\int_{1}^{\frac{\alpha+1}{2}}\Gamma'(x)dx=\int_1^{\frac{\alpha+1}{2}}(\Gamma'(x+\frac{\alpha-1}{2})-\Gamma'(x))dx\\ \le \frac14\max_{y\in(1\wedge\alpha, 1\vee\alpha)}|\Gamma''(y)|(\alpha-1)^2.\end{gathered}   
\end{equation*} 
Since $\Gamma''(1)=\gamma^2+\frac{\pi^2}{6}$, there exists an interval $(1-\rho_2,1+\rho_2)$, in which $|\Gamma''(x)|<2\left(\gamma^2+\frac{\pi^2}{6}\right)$. Taking $\rho=\rho_1\wedge\rho_2$, and $C= \max\left\{2(G_m)^{-2}4\gamma^2, \frac12\left(\gamma^2+\frac{\pi^2}{6}\right)\right\}$, the result is proved. 
    
\end{proof}
\begin{remark} Any attempt to establish similar estimates at point $\alpha=2$ would fail because  both  integrals in the   \eqref{gammasecond} do  not tend to zero, as $\alpha \to 2$. However, as we mentioned before, the point  $\alpha=2$ is not crucial for our process. 
    
\end{remark}

 \section{Power variations of $\mathcal{H}$-fBm}\label{power}  
  Having in hand self-similarity  with index $1/2$, variance that is proportional to $t$ (according to  \eqref{varvar}),  and H\"{o}lder continuity up to order $1/2$ for $\alpha\in(1,2)$ (according to Theorems \ref{thm1} and \ref{thm2}), all these  properties being the same as a Wiener process admits, it is natural to expect that the properties of the $p$-variations  will also be similar to the $p$-variations of the Wiener process. However, the absence of stationarity makes its own adjustments,  as we shall see in  the next Theorems \ref{thm} and \ref{thm2.4}.  Note that the results about existence and properties of $p$-variations considered as the supremum  in all partitions, are sufficiently difficult to investigate in the nonstationary case, therefore  we restrict ourselves to the limit of the power sums according to the sequences of dyadic partitions, and in the mean sense.

\begin{definition} We say that the process $X$ has on the interval $[0,T]$ finite non-zero, zero and infinite $p$-variation if the limit 
 \begin{equation}
    V^p(X, [0,T])=\lim_{n \to \infty}\mathbb{E}\sum_{k=1}^n\left| X\left(\frac{Tk}{n}\right)-X\left(\frac{T(k-1)}{n}\right) \right|^p  . \label{qv}  
\end{equation}  is strictly positive, zero and infinite, respectively.  
\end{definition}
\begin{theorem}\label{thm}
Let $T>0$ and let $B^\mathcal{H}_\alpha$ be the $\mathcal{H }$-fBm with $\alpha \in (1,2)$. Then   $p$-variation $V^p(B^\mathcal{H}_\alpha,  [0,T])$ is finite and non-zero for $p=\frac{2}{\alpha}$, zero for $p>\frac{2}{\alpha}$ and infinite for $p<\frac{2}{\alpha}$.
 In particular, quadratic variation $V^2(B^\mathcal{H}_\alpha,  [0,T])$ is zero (contrary to well known  $V^2(B,  [0,T])=T.$)
\end{theorem}
\begin{proof} (i). Let $p>\frac{2}{\alpha}.$ First, using self-similarity, we can write that $$V^p(B^\mathcal{H}_\alpha, [0,T])=T^{p/2}\lim_{n \to \infty}n^{-p/2}\mathbb{E}\sum_{k=1}^n\left| B^\mathcal{H}_\alpha\left(k\right)-B^\mathcal{H}_\alpha\left( k-1\right) \right|^p. $$  Furthermore, taking into account the relation $$\mathbb{E}|\xi|^p=\left(\mathbb{E}|\xi|^2\right)^{p/2}=\frac{2^{p/2}\Gamma(\frac{p+1}{2})}{\sqrt{\pi}}$$ for a zero-mean Gaussian random variable, it is enough to analyze the  limiting behavior of the following functional: 
$$\widetilde{V}^p(B^\mathcal{H}_\alpha)= \lim_{n \to \infty}n^{-p/2}\sum_{k=1}^n \left(\mathbb{E}\left|B^\mathcal{H}_\alpha\left(k\right)-B^\mathcal{H}_\alpha\left( k-1\right) \right|^2\right)^{p/2}. $$ 
Let us start to analyze the integrals, that, in the previously introduced notations, equal   
\begin{eqnarray}\label{kand k-1}
 J^{(1)}_k:= J^{(1)}_{k-1,k}
 &=&\int _{0}^{k-1} \left[\left( \log \frac{k}{u}\right)^{(\alpha-1)/2} -\left( \log \frac{ k-1}{ u}\right)^{(\alpha-1)/2}\right]^2du\\
& = &   (k-1)\int_0^\infty \left[  \left(\log \frac{k}{k-1}+v  \right)^{\frac{\alpha-1}{2}}-v^{\frac{\alpha-1}{2}}\right]^2e^{-v}dv.\notag
\end{eqnarray}
Now we can apply to the right-hand side of \eqref{kand k-1} the upper bound  from \eqref{qv3},  with $t=k$ and $s=k-1$, and also the inequality $\log(1+x)\le x, \, x>0$, so that we obtain that 
\begin{eqnarray} 
  J^{(1)}_k \le (k-1)\left(\log\left(\frac{k}{k-1}\right)\right)^\alpha\le (k-1)^{1-\alpha}, 
\end{eqnarray}
for $k=2,..,n$, while $J^{(1)}_1 =0.$ 
Therefore, taking into account that    $\sum_{l=1}^{n-1}l^{\beta}\sim n^{\beta+1}$, we arrive at   the relation   $$n^{-p/2}\sum_{k=2}^n(J^{(1)}_k)^{p/2} \leq Cn^{-p/2}\sum_{l=1}^{n-1}l^{(1-\alpha)\frac{p}{2}}\sim  Cn^{1-\frac{\alpha p}{2}} \to 0,$$ 

as $n \to \infty$, if   $\alpha \in (1,2) $ and $p>\frac{2}{\alpha}. $
Now, for the last integral  in \eqref{hol4}, we can write that
\begin{eqnarray}
 J^{(2)}_k:=J^{(2)}_{k-1,k}
 &=&\int _{k-1}^{k} \left( \log \frac{k}{u}\right)^{\alpha-1} du
 =  k\int _{1}^{k/(k-1)} \left( \log z \right)^{\alpha-1} \frac{dz}{z^2} 
  \notag \\
&\leq&  \frac{k}{k-1}\left( \log \frac{k}{k-1} \right)^{\alpha-1} \leq  \left( \frac{1}{k-1} \right)^{\alpha-1}, \notag
\end{eqnarray}
for $k=2,..,n$, and $J^{(2)}_1 $ is a finite constant,  so that $\lim_{n\to\infty}n^{-p/2}\sum_{k=2}^n(J^{(2)}_k )^{p/2}  = 0$,  for $p>\frac{2}{\alpha}$ as above.

(ii). Let $p<\frac{2}{\alpha}.$ It is sufficient to bound from below the sum related to $J^{(2)}_k$, as follows: 
\begin{equation*}\begin{gathered}
 J^{(2)}_k
 =  k\int _{1}^{k/(k-1)} \left( \log z \right)^{\alpha-1} \frac{dz}{z^2}\ge \frac{k(k-1)^2}{k^2}\int _{1}^{k/(k-1)} \left( \log z \right)^{\alpha-1}  dz\\ =\frac{(k-1)^2}{k}\int_0^{\log\left(\frac{k}{k-1}\right)}v^{\alpha-1}e^vdv\ge \frac{(k-1)^2}{\alpha k}\left(\log\left(\frac{k}{k-1}\right)\right)^{\alpha} .    
\end{gathered} 
    \end{equation*} Instead of the upper bound for the logarithm as before, we can apply the Stolz-Ces\`{a}ro theorem, taking into account that $$n^{p/2}-(n-1)^{p/2}\sim n^{p/2}(1-(1-1/n)^{p/2})\sim n^{p/2-1},$$ as $n\to\infty$:

 \begin{equation}\label{equless}\begin{gathered}
n^{-p/2}\sum_{k=2}^n(J^{(2)}_k )^{p/2}\ge n^{-p/2}\sum_{k=2}^n \frac{(k-1)^p}{(\alpha k)^{p/2}}\left(\log\left(\frac{k}{k-1}\right)\right)^{\alpha p/2}\\ \sim \frac{n^{p/2}\left(\log\left(\frac{n}{n-1}\right)\right)^{\alpha p/2}}{n^{p/2}-(n-1)^{p/2}}\sim\frac{n^{p/2}(n-1)^{-\alpha p/2}}{n^{p/2-1}}\sim n^{1-\alpha p/2}\to \infty 
 \end{gathered} 
    \end{equation} as $n\to \infty$, if $p<2/\alpha.$

    (iii). Finally, let $p=2/\alpha.$  Again, applying Stolz-Ces\`{a}ro theorem, we see that the limit
 $$\lim_{n\to\infty}(n^{1/\alpha}-(n-1)^{1/\alpha})^{-1} (J^{(1)}_n+J^{(2)}_n )^{1/\alpha}$$ exists, is finite and nonzero. However, it easily follows from all the previous upper and lower bounds that $J^{(1)}_n+J^{(2)}_n\sim n^{1-\alpha}$ while $(n^{1/\alpha}-(n-1)^{1/\alpha})^{-1}\sim n^{-1/\alpha+1}$, and the result follows. 
\end{proof}
\begin{theorem}\label{thm2.4}
 Let $T>0$ and let $B^\mathcal{H}_\alpha$ be the $\mathcal{H}$-fBm with $\alpha \in (0,1)$. Then the  $p$-variation $V^p(B^\mathcal{H}_\alpha,  [0,T])$ is finite and nonzero for $p=\frac{2}{\alpha}$, zero for $p>\frac{2}{\alpha}$ and infinite for $p<\frac{2}{\alpha}$. In particular, the quadratic variation is infinite.    
\end{theorem}
\begin{proof} Recalling the proof of Theorem \ref{thm} and, in particular, the application of the Stolz-Ces\`{a}ro theorem, it is enough to establish the asymptotic behavior of $$(J^{(1)}_n+J^{(2)}_n )^{p/2}$$ and to compare it with $n^{p/2-1}$. We can rewrite the integral $J^{(1)}_n $, up to a multiplicative constant, according to \eqref{hol4}  (with $t=n$ and $s=n-1$),  in the following way: 
 \begin{eqnarray}
 &&J^{(1)}_n  
 \sim  \int _{0}^{n-1} \left[\left( \log \frac{n}{u}\right)^{(\alpha-1)/2} -\left( \log \frac{n-1}{u}\right)^{(\alpha-1)/2}\right]^2du \notag\\
&\sim& n\int _{0}^{+\infty}\frac{e^{-z}\left[\left( \log \frac{n}{n-1}+z\right)^{(1-\alpha)/2} -z^{(1-\alpha)/2}\right]^2}{z^{1-\alpha}\left(\log \frac{n}{n-1}+z\right)^{1-\alpha}}dz , \notag
\\
&=& n\rho^\alpha\int _{0}^{+\infty}\frac{e^{-\rho v}\left[\left( 1+v\right)^{(1-\alpha)/2} -v^{(1-\alpha)/2}\right]^2}{v^{1-\alpha}\left(1+v\right)^{1-\alpha}}dv, \notag
\end{eqnarray}
 where $\rho=\log(n/n-1)$. Obviously, $$\int _{0}^{+\infty}\frac{e^{-\rho v}\left[\left( 1+v\right)^{(1-\alpha)/2} -v^{(1-\alpha)/2}\right]^2}{v^{1-\alpha}\left(1+v\right)^{1-\alpha}}dv\to \int _{0}^{+\infty}\frac{ \left[\left( 1+v\right)^{(1-\alpha)/2} -v^{(1-\alpha)/2}\right]^2}{v^{1-\alpha}\left(1+v\right)^{1-\alpha}}dv \in (0,\infty),$$ 
 as $n \to \infty$; therefore $$J^{(1)}_n \sim n\rho^\alpha\sim n^{1-\alpha}.$$
 Furthermore, $J^{(2)}_n$ is estimated similarly to the case $\alpha\in(1,2)$ and its asymptotic behavior is $J^{(2)}_n\sim n^{1-\alpha}.$ Finally, $$(J^{(1)}_n+J^{(2)}_n)^{p/2}\sim n^{(1-\alpha)p/2},$$
 and $$V^p(B^\mathcal{H}_\alpha,  [0,T])\sim n^{(1-\alpha)p/2}n^{-p/2+1}=n^{1-\alpha p/2}.$$ We thus obtained the same asymptotics as for the case $\alpha\in(1,2)$, and the theorem is proved.  
\end{proof}
\section{Local nondeterminism of $\mathcal{H}$-fBm}\label{locnondeterm}
Now we shall prove  that $\mathcal{H}$-fBm  is locally nondeterministic on any open interval $(0,T)$,     $T>0$. Recall that, according to \cite{Berman}, a zero mean Gaussian process $\{X(t)\}_{t\in\mathbb{R}}$ is  called  locally nondeterministic  (LND) on some interval $\mathbb{T}=(a,b)$ if $ X$ satisfies   the following three assumptions:
\begin{itemize}\label{condition A}
\item[ $(i)$] $\mathbb{E}X^2(t)>0$ for all $t\in \mathbb{T}$;
\item[$(ii)$] $\mathbb{E} (X(t)-X(s))^2 >0$ for all $t,s\in \mathbb{T}$;
\item[$(iii)$] for any $m\geq 2$,
\begin{equation}\label{BermanVM}
\liminf_{\epsilon\downarrow 0}V_{m}=\frac{{\text{Var}}(X(t_m)-X(t_{m-1})|X(t_1), \ldots, X(t_{m-1}))}{{\text{Var}}(X(t_m)-X(t_{m-1}))}>0,
\end{equation}
where the $\liminf$ is taken over distinct, ordered $t_1<t_2<\ldots<t_m\in (a,b)$ with $|t_1-t_m|<\epsilon$.
\end{itemize}
\begin{theorem}
Let $\alpha\in(0,1)\cup(1,2)$. Then  $\mathcal{H}$-fBm  $B^\mathcal{H}_\alpha$ is locally nondeterministic on any interval $(0,T), \,T>0$.  
\end{theorem}
\begin{proof} Since $B^\mathcal{H}_\alpha$ is a Gaussian process, its conditional variance  $$V_1:={\text{Var}}( B^\mathcal{H}_\alpha(t_m)- B^\mathcal{H}_\alpha(t_{m-1})| B^\mathcal{H}_\alpha(t_1), \ldots,  B^\mathcal{H}_\alpha(t_{m-1}))$$ is non-random and equals
$$V_1=\mathbb{E}( B^\mathcal{H}_\alpha(t_m)- B^\mathcal{H}_\alpha(t_{m-1}))^2-\mathbb{E}(\mathbb{E}( B^\mathcal{H}_\alpha(t_m)- B^\mathcal{H}_\alpha(t_{m-1})| B^\mathcal{H}_\alpha(t_1), \ldots,  B^\mathcal{H}_\alpha(t_{m-1})))^2.$$
 Furthermore, since, for any $t>0$,  $\sigma(B^\mathcal{H}_\alpha(s),\,0\le s\le t)\subset \sigma(B(s),\,0\le s\le t),$ then
 we can state that 
 \begin{eqnarray}&&
 \mathbb{E}( B^\mathcal{H}_\alpha(t_m)- B^\mathcal{H}_\alpha(t_{m-1})| B^\mathcal{H}_\alpha(t_1), \ldots,  B^\mathcal{H}_\alpha(t_{m-1})) \notag \\&=& \mathbb{E}( \mathbb{E}(B^\mathcal{H}_\alpha(t_m)- B^\mathcal{H}_\alpha(t_{m-1})| B_s, 0\le s\le t_{m-1})|B^\mathcal{H}_\alpha(t_1), \ldots,  B^\mathcal{H}_\alpha(t_{m-1})) \notag \\&=&\frac{1}{\sqrt{\Gamma(\alpha)}}\int_{0}^{t_{m-1}}\left(\left(\log\left(\frac{t_m}{u}\right)\right)^{\frac{\alpha-1}{2}}-\left(\log\left(\frac{t_{m-1}}{u}\right)\right)^{\frac{\alpha-1}{2}}\right)dB(u),  \notag\end{eqnarray}
      whence $$V_1=\frac{1}{\Gamma(\alpha)}\int_{t_{m-1}}^{t_m}  \left(\log\left(\frac{t_m}{u}\right)\right)^ {\alpha-1}du.$$
      Recall that $$V_2:={\text{Var}}(B^\mathcal{H}_\alpha(t_m)-B^\mathcal{H}_\alpha(t_{m-1}))=V_1+ V_3,$$ where $$V_3=\frac{1}{\Gamma(\alpha)}\int_{0}^{t_{m-1}}\left(\left(\log\left(\frac{t_m}{u}\right)\right)^{\frac{\alpha-1}{2}}-\left(\log\left(\frac{t_{m-1}}{u}\right)\right)^{\frac{\alpha-1}{2}}\right)^2du,$$
      and therefore, we have to prove that  $\frac{V_1}{V_2}=\frac{V_1}{V_1+V_3}$ is separated from zero, for which  it is sufficient to prove that $\frac{V_3}{V_1}$ is bounded. Now consider two cases. If $\alpha\in(1.2)$, then,   according to \eqref{qv3} with $s=t_{m-1}$ and $t=t_m,$  $$V_3\le \frac{K_\alpha}{\Gamma(\alpha)} t_{m-1}\left(\log \frac{t_m}{t_{m-1}}  \right)^\alpha,$$ while 
      \begin{equation}\begin{gathered}\label{lowerbo}V_1= \frac{1}{\Gamma(\alpha)}\int_0^{\log(t_m/t_{m-1})}z^{\alpha-1}e^{-z}d\ge  \frac{1}{ \Gamma(\alpha+1)}t_m\left(\log \frac{t_m}{t_{m-1}}  \right)^\alpha e^{-\log(t_m/t_{m-1}})\\=\frac{1}{ \Gamma^2(\alpha+1)}t_{m-1}\left(\log \frac{t_m}{t_{m-1}}  \right)^\alpha,\end{gathered}\end{equation}
      and the proof immediately follows. Let $\alpha\in(0,1).$ Then we can get the same (up to a constant multiplier) upper bound for $V_3$ from \eqref{needlater} and  \eqref{first}, while the lower bound \eqref{lowerbo} for $V_1$ is good for $\alpha\in(0,1) $ as well, so that the theorem is proved. 
\end{proof}

\section{The integral with respect to the $\mathcal{H}$-fBm}\label{integral}

\subsection{Preliminary results}


We will make use of the following properties holding for the Mellin transform, defined as  $\mathfrak{M}_s(f):=\int_0^\infty x^{s-1}f(x)dx$, for $s \in \mathbb{C}$.

\begin{lem}\label{pars}(\textit{Parseval formula for Mellin transform} (\cite{TIT}))

Let $\mathfrak{M}_s(f_i)$ be the Mellin transform of the function $f_i:\mathbb{R}^+ \to \mathbb{R}$, with strip of holomorphy $S_{f_i}:=\left\{s:\eta_{i}<\Re(s)<\eta'_i\right\}$
, $i=1,2$. Let moreover $x^{c-1/2}f_1(x)\in L_2(\mathbb{R}^+)$, for $c \in \mathbb{R}$, with $\eta_1 <c<\eta'_1$. Then, if $x^{s_0-c-1/2}f_2(x)\in L_2(\mathbb{R}^+)$, for $s_0 \in \mathbb{C},$ with $\eta_2<\Re(s_0)-c<\eta_2'$, we have that
\begin{equation}
\int_0^{\infty}f_1(t)f_2(t)t^{s_0 -1}dt=\frac{1}{2 \pi i}\int_{c-i \infty}^{c+i \infty}\mathfrak{M}_s(f_1)\mathfrak{M}_{s_0-s}(f_2)ds,
\end{equation}
where the last integral is on the vertical line $\Re(s)=c.$
\end{lem}
\begin{lem}\label{planch}(\textit{Plancherel formula for Mellin transform} (\cite{TIT}))

Let $\mathfrak{M}_s(f)$ be the Mellin transform of the function $f:\mathbb{R}^+ \to \mathbb{R}$, with strip of holomorphy $S_{f}:=\left\{s:\eta<\Re(s)<\eta'\right\}$. Then, if $x^{c-1/2}f(x)\in L_2(\mathbb{R}^+)$, for $c \in \mathbb{R},$ with $\eta<c<\eta'$, we have that
\begin{equation}
    \int_0^{\infty}|f(x)|^2x^{2c-1}dx=\frac{1}{2 \pi}\int_{- \infty}^{+ \infty}\left\vert \mathfrak{M}_{c+it}(f)\right\vert^2 dt. \label{pla}
\end{equation}

\end{lem}
We recall that the Mellin transform of the Hadamard fractional integral $_{\mathcal{H}}\mathcal{I}_{-}^{\beta }$ is given by $\mathfrak{M}_s\left( _{\mathcal{H}}\mathcal{I}_{-}^{\beta }f\right)=s^{-\beta}\mathfrak{M}_s\left( f \right)$, while, for the fractional derivative, we have that  $\mathfrak{M}_s\left( _{\mathcal{H}}\mathcal{D}_{-}^{\beta }f\right)=s^{\beta}\mathfrak{M}_s\left( f \right)$, for $\Re(s)>0$ (see \cite{KIL}, p.119-120).

We will refer to the following space: let $f:\mathbb{R}^+ \to \mathbb{R}$ have strip of holomorphy $S_{f} $ such that $1/2 \in \mathcal{S}_f$; then $x^{c-1/2}f(x)\in L_2(\mathbb{R}^+)$, for $c=1/2$, if $f \in L_2(\mathbb{R}^+)$. We define, for $\beta>0,$
\[\mathfrak{F}_{\beta}:=\left\{f \in L_2(\mathbb{R}^+), f :\mathbb{R}^+ \to \mathbb{R}\left\vert \frac{1}{ 2\pi }\int^{+ \infty}_{- \infty}\left\vert\mathfrak{M}_{1/2+it}(f) \right\vert^2 \left(\frac{1}{4}+t^2\right)^\beta  dt < \infty \right.\right\},\]
with the norm \begin{eqnarray}
\Vert f\Vert^2_{\mathfrak{F}_{\beta}}&:=&\frac{1}{ 2\pi }\int^{+ \infty}_{- \infty}\left\vert\mathfrak{M}_{1/2+it}(f) \right\vert^2 \left(\frac{1}{4}+t^2\right)^\beta  dt =\frac{1}{ 2\pi }\int^{+ \infty}_{- \infty}\left\vert\mathfrak{M}_{1/2+it}(f) \right\vert^2 \left\vert\frac{1}{2}+it\right\vert^{2\beta}  dt \notag \\
 &=&\frac{1}{ 2\pi }\int^{+ \infty}_{- \infty}\left\vert\mathfrak{M}_{1/2+it}\left(_{\mathcal{H}}\mathcal{D}_{-}^{\beta }f\right) \right\vert^2   dt =\int_0^\infty \left( _{\mathcal{H}}\mathcal{D}_{-}^{\beta }f\right)^2(t)dt, \label{fspace}\end{eqnarray} 
where, in the last step, we applied Lemma \ref{planch}, with $c=1/2$. It follows from \eqref{fspace} that $\mathfrak{F}_{\beta}$ is a normed linear space,
endowed by the inner product
\begin{eqnarray}
\langle f_1,f_2\rangle_{\mathfrak{F}_{\beta}}&:=& \frac{1}{ 2\pi i}\int^{1/2+i \infty}_{1/2-i \infty}\mathfrak{M}_s\left(f_1\right)  \mathfrak{M}_{1-s}\left(f_2\right) s^\beta (1-s)^{\beta}ds\\
&=&
\frac{1}{ 2\pi i}\int^{1/2+i \infty}_{1/2-i \infty}\mathfrak{M}_s\left(_{\mathcal{H}}\mathcal{D}_{-}^{\beta }f_1\right)  \mathfrak{M}_{1-s}\left(_{\mathcal{H}}\mathcal{D}_{-}^{\beta }f_2\right) ds 
=\int_0^{+\infty}\left(_{\mathcal{H}}\mathcal{D}_{-}^{\beta }f_1\right)(x)\left(_{\mathcal{H}}\mathcal{D}_{-}^{\beta }f_2\right)(x)dx \notag
\end{eqnarray}
where, in the last step, we have applied Lemma \ref{pars}, with $c=1/2$ and $s_0=1$.
Hereafter, we will refer to $\mathfrak{F}_{\beta}$, in the case $\beta =(1-\alpha)/2$, for $\alpha \in (0,1)$, i.e. with $\beta \in (0,1/2)$. 

We then prove the following preliminary result, where we refer to step (or elementary) functions defined as $f^*(x):=\sum_{j=1}^{k}c_j 1_{A_j}(x),$ $x \in \mathbb{R},$ $A_j \in \mathcal{B}(\mathbb{R})$ and $c_j \in \mathbb{R}$, for $k \in \mathbb{N}$.

\begin{lem}\label{lemstep}
 For any function $g: \mathbb{R}^+ \to \mathbb{R}$, such that $g \in \mathfrak{F}_{\beta}$, for $\beta \in (0,1/2)$, there exists a sequence of step functions $\left\{g^*_n\right\}_{n \geq 1}$, with $g^*_n \in \mathfrak{F}_{\beta}$, for any $n$, such that
\begin{equation}\Vert g-  g^*_n \Vert_{\mathfrak{F}_{\beta}}\to 0, \qquad n \to \infty.\label{step3}   \end{equation}    
\end{lem}
\begin{proof}
We resort to the following result proved in \cite{PIP}, Lemma 5.1 (see also \cite{PIP2}), for the space
\[
\mathcal{F}_{\beta}:=\left\{f \in L_2(\mathbb{R}), f :\mathbb{R} \to \mathbb{R}\left\vert \int^{+ \infty}_{- \infty}\left\vert\tilde{f}(y) \right\vert^2 y^{2\beta}  dy < \infty\right. \right\},
\]
where $\tilde{f}(y):=\int_{-\infty}^{+\infty}e^{iyx}f(x)dx$, for $y \in \mathbb{R}$ and $\beta \in (0,1/2)$: for any function $f \in \mathcal{F}_{\beta}$ there exists a sequence of elementary functions $\left\{f^*_n\right\}_{n \geq 1}\in L_2(\mathbb{R}),$ such that $\Vert f^*_n - f\Vert_{\mathcal{F}_{\beta}}\to 0$, as $n \to \infty$, with the norm $\Vert f\Vert^2_{\mathcal{F}_\beta} =\int^{+ \infty}_{- \infty}|\tilde{f}(y) |^2 y^{2\beta}  dy$. We show that it holds also for the (slightly modified) space 
\[
\mathcal{F}'_{\beta}:=\left\{f \in L_2(\mathbb{R}), f :\mathbb{R} \to \mathbb{R}\left\vert \int^{+ \infty}_{- \infty}\left\vert\tilde{f}(y) \right\vert^2 \left(\frac{1}{4}+y^{2}\right)^\beta  dy < \infty\right. \right\},
\]
by considering that the measure $\lambda_\beta(dx):=|1/4+x^2|^{\beta}|x|^{-2}dx$, is finite around $|x|=\infty$, for $\beta \in (0,1/2)$. For any $g:\mathbb{R}^+ \to \mathbb{R}$ with $g \in \mathfrak{F}_{\beta}$ we rewrite \eqref{fspace} as follows
\begin{equation}
\Vert g\Vert^2_{\mathfrak{F}_{\beta}}=\frac{1}{ 2\pi }\int^{+ \infty}_{- \infty}\left\vert\mathfrak{M}_{1/2+it}(g) \right\vert^2 \left(\frac{1}{4}+t^2\right)^\beta  dt =\frac{1}{ 2\pi }\int^{+ \infty}_{- \infty}\left\vert\int_{-\infty}^{+\infty}e^{ity-y/2}g(e^y))dy \right\vert^2 \left(\frac{1}{4}+t^2\right)^\beta   dt.\notag\end{equation}
Then, if we denote $f(y):=e^{-y/2}g(e^y)$, for $y \in \mathbb{R}$, then $f \in \mathcal{F}'_\beta$ and we prove that there exists a sequence of step functions approximating $f$ in $ \mathcal{F}'_{\beta}$. As in the proof of the above mentioned Lemma 5.1 in \cite{PIP}, we can consider first the case where $f(\cdot)$ is even (i.e. when $g(e^{-y})=e^{-y}g(e^y),$ $\omega \in \mathbb{R}$) and thus $\tilde{f}(t)=\int_{-\infty}^{+\infty}e^{ity-y/2}g(e^y)dy$ is real-valued and even; let $k>1$ such that, for $\epsilon>0,$ \begin{equation}\int_\mathbb{R}\left(\frac{1}{4}+t^2\right)^\beta |t|^{-2}1_{\left\{|t|>k\right\}}dt<\epsilon^2/2.\label{tog2}\end{equation}

 Then,  proving that there exists a step function $f^*(\cdot)$ in $\mathcal{F}'_\beta$ such that it is $\Vert f^* - f\Vert_{\mathcal{F}'_{\beta}}< \epsilon$ reduces to constructing a step function $l^*:\mathbb{R} \to \mathbb{R}$, belonging to $\mathcal{F}'_\beta$, for which \begin{equation}
\int^{+ \infty}_{- \infty} \left\vert1_{[-1,1]}(t)-\tilde{l^*}(t)\right\vert^2\left(\frac{1}{4}+t^2\right)^\beta dt=\int^{+ \infty}_{- \infty} \left\vert t1_{[-1,1]}(t)-t\tilde{l^*}(t)\right\vert^2\lambda_\beta(dt)<\epsilon^2.\label{um2}\end{equation} We define the function $U:\mathbb{R} \to \mathbb{R}$ which equals $t1_{[-1,1]}(t)$ on $[-k,k]$ and is periodically extended for $t \in \mathbb{R}.$ It has Fourier series expansion $\sum_{n=-\infty}^{+\infty}
u_ne^{i\pi n t/k}$, which converges to
$U$ everywhere on $[-k, k]$, except at the points $t = \pm 1$ where $U$ is discontinuous. Moreover, the partial sum  $U_m(t):= \sum_{n=-m}^{m}
u_ne^{i\pi n t/k}$ can be expressed as
\begin{equation}U_m(t)=\frac{1}{k}\int_{-k}^kU(t-y)D_m(\pi y/k)dy, \label{um}\end{equation}
where $D_m(y):=\sin(m+1/2)y/(2\sin(y/2))$ is the Dirichlet kernel. As proved in \cite{PIP}, $U_m$ given in \eqref{um} satisfies the following conditions: $(i)$ $\sup_m \sup_t |U_m(t)| \leq const.$ and $(ii)$ $\sup_m  |U_m(t)| \leq |t| const$. Thus, we can apply the dominated convergence theorem, so that
there exists an integer $M$ such that
\begin{equation}\int^{+ \infty}_{- \infty} \left\vert t1_{[-1,1]}(t)-U_M(t)\right\vert^2\lambda_\beta(dt)< \epsilon^2/2.\label{tog}\end{equation}
By recalling that $x^{-1}U_M(x)$ is the Fourier transform of the step function $l^*(x):=\sum_{m=1}^M a_n 1_{[-\pi n/k,\pi n/k]}(x)$, where $a_n \in \mathbb{R}$ are given by $a_n:= i u_n sign(n)$, for $|n| \geq 1$,
  we get the final result \eqref{um2}, since we can write
\begin{eqnarray}
&&\int^{+ \infty}_{- \infty} \left\vert1_{[-1,1]}(t)-\tilde{l^*}(t)\right\vert^2\left(\frac{1}{4}+t^2\right)^\beta dt=\int^{+ \infty}_{- \infty} \left\vert t1_{[-1,1]}(t)-t\tilde{l^*}(t)\right\vert^2\lambda_\beta(dt) \notag \\
&\leq&\int_{|t| \leq k} \left\vert t1_{[-1,1]}(t)-t\tilde{l^*}(t)\right\vert^2\lambda_\beta(dt)+\int_{|t| > k} \lambda_\beta(dx) <\epsilon^2, \label{ind1}
\end{eqnarray}
by considering together \eqref{tog2} and \eqref{tog}. The case of odd $f(\cdot)$ can be treated analogously (see \cite{PIP} for details). Then, by denoting $
g^+(x):=\sqrt{x}l^*(\log(x))=\sqrt{x}\sum_{n=1}^M a_n 1_{[e^{-\pi n/k},e^{\pi n/k}]}(x)$ , $x \in \mathbb{R}^+$, we can write that,
for any $\epsilon>0,$
\[\Vert g^+ - g\Vert_{\mathfrak{F}_{\beta}}=\Vert \sqrt{\cdot}[l^*(\log \cdot) - f(\log \cdot)]\Vert_{\mathfrak{F}_{\beta}}=2 \pi\Vert l^* - f\Vert_{\mathcal{F}'_{\beta}}<\epsilon.\] 
It remains to prove that there exists a sequence of step functions $\left\{\varphi_n^*\right\}_{n \geq 1} \in \mathfrak{F}_{\beta}$ that approximates $g^+$ in $\Vert\cdot\Vert_{\mathfrak{F}_{\beta}}$ or, for the sake of simplicity, that, for $0< a<b$, $\Vert \sqrt{\cdot} 1_{[a,b]}(\cdot)-\varphi_n^*(\cdot)\Vert^2_{\mathfrak{F}_{\beta}} \to 0$, as $n \to \infty$.  Let us define, for any $n \in \mathbb{N}$, $\varphi_n^*:=\sum_{k=1}^n\sqrt{x_k}1_{[x_{k-1},x_k]}$, where $x_k:=a+k(b-a)/n$, for $k=1,2,...,n$. Then, we can write that, for any $n$,
\begin{eqnarray}
   && \Vert \sqrt{\cdot} 1_{[a,b]}(\cdot)-\varphi_n^*(\cdot)\Vert^2_{\mathfrak{F}_{\beta}}=\frac{1}{ 2\pi }\int^{+ \infty}_{- \infty}\left\vert\mathfrak{M}_{1/2+it}(\sqrt{\cdot} 1_{[a,b]}(\cdot)-\varphi_n^*(\cdot)) \right\vert^2 \left(\frac{1}{4}+t^2\right)^\beta  dt \notag \\
    &=& \frac{1}{ 2\pi }\int^{+ \infty}_{- \infty}\left\vert\sum_{k=1}^n\int_{x_{k-1}}^{x_k} \omega^{it-1/2}(\sqrt{\omega}-\sqrt{x_k})d\omega \right\vert^2 \left(\frac{1}{4}+t^2\right)^\beta  dt \notag \\
  &=&\frac{1}{ 2\pi }\int^{+ \infty}_{- \infty}
    \left\vert \sum_{k=1}^n\left[\frac{\omega^{it+1/2}}{it+1/2}(\sqrt{\omega}-\sqrt{x_k}) \right]_{x_{k-1}}^{x_k}-\frac{1}{2it+1}\int_a^{b} \omega^{it} d\omega \right\vert^2 \left(\frac{1}{4}+t^2\right)^\beta  dt \notag \\
     &=&\frac{1}{ 2\pi }\int^{+ \infty}_{- \infty}
    \left\vert \sum_{k=1}^n x_{k-1}^{it+1/2}\frac{\sqrt{x_k}-\sqrt{x_{k-1}}}{it+1/2} -\frac{1}{2}\frac{ b^{it+1}-a^{it+1}}{(it+1/2)(it+1)} \right\vert^2 \left(\frac{1}{4}+t^2\right)^{\beta}  dt \notag\\
    &\leq&\frac{1}{ 2\pi }\int^{+ \infty}_{- \infty}
    \left[ b+\sqrt{ab} +\frac{1}{2}\frac{ b+a}{1+t^2} \right]^2 \left(\frac{1}{4}+t^2\right)^{\beta-1}  dt \leq \frac{1}{ 2\pi }\int^{+ \infty}_{- \infty}
    \left[ 2b+\frac{ b}{1+t^2} \right]^2 \left(\frac{1}{4}+t^2\right)^{\beta-1}  dt. \notag
\end{eqnarray}
The last integral is clearly finite, so that, by using the dominated convergence theorem and by noting that $\left\vert\mathfrak{M}_{1/2+it}(\sqrt{\cdot} 1_{[a,b]}(\cdot)-\varphi_n^*(\cdot)) \right\vert^2 \left(\frac{1}{4}+t^2\right)^\beta$ converges to zero pointwise in $t$, as $n \to \infty$, the result follows.
\end{proof}


Our aim here is to introduce the Wiener integral of a function, with respect to the $\mathcal{H}$-fBm.

Let us define the inner product on the space $L_{2,\alpha}(\mathbb{R}^+):=\left\{ f: \;  _\mathcal{H}\mathcal{M}_-^\alpha f \in L_2(\mathbb{R}^+) \right\}$ as 
\begin{equation}\label{ip}\langle f,g \rangle_{L_{2,\alpha}}:=\int_0^\infty {_{\mathcal{H}}\mathcal{M}_-^\alpha} f(x) {_{\mathcal{H}}\mathcal{M}_-^\alpha} g(x) dx= \langle \; _\mathcal{H}\mathcal{M}_-^\alpha f, \; _\mathcal{H}\mathcal{M}_-^\alpha g\rangle_{L_{2}},\end{equation}
where we omitted in the subscripts, for brevity, the indication of $\mathbb{R}^+$, for any $\alpha \geq 0.$ Note that, for $\alpha \in (0,1)$, thanks to \eqref{fspace},  the space   $L_{2,\alpha}(\mathbb{R}^+)$ coincides with $\mathfrak{F}_{(1-\alpha)/2}$.
\subsection{Construction of the Wiener integral w.r.t. $\mathcal{H}$-fBm}\label{integinteg}
Recalling the compact interval representation  \eqref{mvn} 
  of the $\mathcal{H}$-fBm, as the Wiener integral w.r.t a Wiener process with a Volterra kernel,  we give the following definition. 
 
\begin{definition}\label{defibeta}
    Let $\Vert f \Vert^2_{L_{2,\alpha}}:=\langle f,f\rangle_{L_{2,\alpha}}$ be the norm defined on $L_{2,\alpha}(\mathbb{R}^+)$ and let $\;_{\mathcal{H}}\mathcal{M}_-^\alpha$ be the operator introduced in \eqref{malf}, then we define the following stochastic integral w.r.t. the $\mathcal{H}$-fBm
    \begin{equation}\label{ibeta}
    I_\alpha(f):= \int_{\mathbb{R}^+}f(s)dB^\mathcal{H}_\alpha(s)=\int_{\mathbb{R}^+}\left(\;_{\mathcal{H}}\mathcal{M}_-^\alpha f\right)(s)dB(s).
    \end{equation}
   
\end{definition}
    
We recall that in \cite{BEG}, the existence of the noise of the $\mathcal{H}$-fBm was proved in the more general setting where the Le Roy measure is used, instead of the Gaussian measure $\nu(\cdot)$, which is a special case of the previous one. We now prove that the first integral in \eqref{ibeta} is well-defined as the limit of the corresponding integral for an approximating sequence of step functions.

\bigskip

Let us define for a step function $f: \mathbb{R}^+ \to \mathbb{R}$, such that $f(t):=\sum_{k=1}^n a_k 1_{[t_{k-1},t_k)}(t)$, for $t_0<t_1<...<t_n \in \mathbb{R}^+$, $a_k \in \mathbb{R},$ $k=1,...,n.$, the integral w.r.t. the $\mathcal{H}$-fBm as:
    \begin{eqnarray}
        \mathcal{J}_\alpha(f)&: =&\sum_{k=1}^n a_k\int_{\mathbb{R}^+}\left(\;_{\mathcal{H}}\mathcal{M}_-^\alpha 1_{[t_{k-1},t_k)}\right)(s)dB(s) \notag \\
        &=&\frac{K_\alpha}{\Gamma((\alpha+1)/2)}\sum_{k=1}^n a_k\left[\int_0^{t_k}\left(\log \frac{t_k}{s}\right)^{(\alpha-1)/2} dB(s)-\int_0^{t_{k-1}}\left(\log \frac{t_{k-1}}{s}\right)^{(\alpha-1)/2} dB(s)\right] \notag \\
        &=&\sum_{k=1}^n a_k[B^\mathcal{H}_\alpha(t_k)-B^\mathcal{H}_\alpha(t_{k-1})], \notag
    \end{eqnarray}
    where, and in the first step, we have applied Lemma 2.1 in \cite{BEG}, whilst, in the last one, we used $K_{\alpha}=\Gamma ((\alpha+1)/2)/\sqrt{\Gamma(\alpha)}$ together with \eqref{mvn}.
Thus, we have that
\begin{eqnarray}
&&\Vert \mathcal{J}_\alpha(f)\Vert_{L_2}^2 \notag \\
&=&\sum_{i,k=1}^n a_ia_k\int_{\mathbb{R}^+}\left(\;_{\mathcal{H}}\mathcal{M}_-^\alpha 1_{[t_{k-1},t_k)}\right)(s)\left(\;_{\mathcal{H}}\mathcal{M}_-^\alpha 1_{[t_{i-1},t_i)}\right)(s)ds \notag \\
&=&\frac{1}{\Gamma^2(\alpha)} \sum_{i,k=1}^n a_ia_k\int_{\mathbb{R}^+}\left[  \left(\log\frac{t_i}{s} \right)_+^{(\alpha-1)/2}-\left(\log\frac{t_{i-1}}{s} \right)_+^{(\alpha-1)/2}\right]\left[  \left(\log\frac{t_k}{s} \right)_+^{(\alpha-1)/2}-\left(\log\frac{t_{k-1}}{s} \right)_+^{(\alpha-1)/2}\right]ds \notag \\
&=&\frac{1}{\Gamma^2(\alpha)} \sum_{i,k=1}^n a_ia_k\left[A^{(i,k)}_+-A^{(i,k-1)}_+-A^{(i-1,k)}_+ +A^{(i-1,k-1)}_+\right], \notag
\end{eqnarray}
where $A^{(i,k)}_+=C_\alpha t_i \Psi\left(-(\alpha-1)/2, -(\alpha-1) ; \log\left(\frac{t_k}{ t_i}\right) \right)1_{t_i \leq t_k}$.

 We now prove that we can approximate any $f \in L_{2,\alpha}(\mathbb{R}^+)$ by a sequence of step functions $f_n \in L_{2,\alpha}(\mathbb{R}^+)$, for $n \geq 1.$ Our result is analogous to the statements proved for fractional Riemann-Liouville operators and, correspondingly, for fractional Brownian motion, first in \cite{BEN}, and then presented  in   \cite{MIS}, Lemma 1.6.2.

    \begin{theorem}\label{thmden}
     For $\alpha \in (0,1)\cup(1,2)$ the linear span of the set $\left\{_{\mathcal{H}}\mathcal{M}^\alpha_- 1_{(a,b)}; a,b \in \mathbb{R}^+ , a<b\right\}$ is dense in $L_2(\mathbb{R}^+)$. 
    \end{theorem}
    \begin{proof}
     \textit{(i)}  We start with the case $\alpha \in (0,1)$, i.e. with $_\mathcal{H}\mathcal{M}_{-}^{\alpha }=K_\alpha\thinspace_\mathcal{H}\mathcal{D}_{-}^{(1-\alpha)/2 }$; 
by recalling Lemma \ref{lem2.1}, it is easy to check that, for $\alpha \in (0,1)$ and 
$0 \leq a<b,$%
\begin{equation}\label{funct}
f(x):=\left( _\mathcal{H}\mathcal{D}_{-}^{(1-\alpha)/2 }1_{[a,b)}\right) (x)=\frac{1}{\Gamma
((\alpha+1)/2 )}\left[ \left( \log \frac{b}{x}\right) _{+}^{(\alpha-1)/2 }-\left(
\log \frac{a}{x}\right) _{+}^{(\alpha-1)/2}\right]  
\end{equation}%
has $S_f=\left\{s: 0<\Re(s)<\infty \right\}$; moreover, $x^{c-1/2}f(x)\in L_2(\mathbb{R}^+)$, for $c=1/2$, so that the Plancherel formula given in Lemma \ref{planch} holds and
\begin{equation}
\frac{1}{ 2\pi }\int^{+ \infty}_{- \infty}\left\vert\mathfrak{M}_{1/2+it}(1_{(a,b)}) \right\vert^2 \left(\frac{1}{4}+t^2\right)^\beta  dt
=\int_0^{\infty}f^2(x)dx < +\infty.
\end{equation}
Thus $1_{(a,b)} \in\mathfrak{F}_{(1-\alpha)/2}$ and the same holds for any step function $\varphi^*_n(x):=\sum_{k=1}^n a_k 1_{[x_{k-1},x_k)}(x)$, for $x_0<x_1<...<x_n \in \mathbb{R}^+$, $a_k \in \mathbb{R},$ $k=1,...,n$, and $n \geq 1$.

Now, by considering again Lemma \ref{planch}, in order to obtain the final result, it is sufficient to prove
that the linear span of the functions \begin{equation}
    \mathfrak{M}_s\left( _{\mathcal{H}}\mathcal{D}_{-}^{(1-\alpha)/2 }1_{[a,b)}\right)=s^{(1-\alpha)/2}\mathfrak{M}_s\left( 1_{[a,b)} \right)=(b^s-a^s)s^{-(1+\alpha)/2}\label{step2}
    \end{equation}
    is dense in $L_2(\mathbb{R}^+)$, i.e. that, for any function $f \in L_2(\mathbb{R}^+)$, there exists a sequence of step functions $f^*_n  \in S_f$ such that
\begin{equation}
\mathscr{I}_\alpha(f,f^*_n):=\frac{1}{2 \pi }\int_{- \infty}^{+\infty }\left\vert \mathfrak{M}_{1/2+it}(f)-\mathfrak{M}_{1/2+it}(f^*_n)(1/2+it)^{(1-\alpha)/2}\right\vert^2 dt \to 0. \notag
 \end{equation}
Since $\mathfrak{M}_{s}\left( _{\mathcal{H}}\mathcal{I}_{-}^{\beta }f\right)=s^{-\beta}\mathfrak{M}_{s}(f), $ $s \in \mathbb{C}$ with $\Re(s)>0,$ we can rewrite the previous integral as follows
 \begin{eqnarray}
 \mathscr{I}_\alpha(f,f^*_n)&=&
\frac{1}{2 \pi }\int_{- \infty}^{+\infty} \left\vert (1/2+it)^{(\alpha-1)/2}\mathfrak{M}_{1/2+it}(f)-\mathfrak{M}_{1/2+it}(f^*_n)\right\vert^2 |1/2+it|^{1-\alpha}dt \notag \\
&=&\frac{1}{2 \pi }\int_{- \infty}^{+\infty} \left\vert \mathfrak{M}_{1/2+it}\left( _{\mathcal{H}}\mathcal{I}_{-}^{(1-\alpha)/2 }f\right)-\mathfrak{M}_{1/2+it}(f^*_n)\right\vert^2 |1/2+it|^{1-\alpha}dt \notag
\end{eqnarray}
which goes to zero, as $n \to \infty,$ by recalling Lemma \ref{lemstep} (for $\beta=(1-\alpha)/2$) and by considering that  $_{\mathcal{H}}\mathcal{I}_{-}^{(1-\alpha)/2 } f\in \mathfrak{F}_{(1-\alpha)/2}$, since
\begin{eqnarray}
&&\frac{1}{ 2\pi }\int^{+ \infty}_{- \infty}\left\vert\mathfrak{M}_{1/2+it}\left( _{\mathcal{H}}\mathcal{I}_{-}^{(1-\alpha)/2 }f\right) \right\vert^2 \left(\frac{1}{4}+t^2\right)^{(1-\alpha)/2}  dt \notag \\
&=&\frac{1}{ 2\pi }\int^{+ \infty}_{- \infty}\left\vert\mathfrak{M}_{1/2+it}(f) \right\vert^2 \left\vert\frac{1}{2}+it\right\vert^{2\alpha-2} \left(\frac{1}{4}+t^2\right)^{1-\alpha} dt = \int_0^\infty f(t)^2 dt <\infty. \notag
\end{eqnarray}


      \textit{(ii)} For $\alpha \in (1,2)$, i.e. for $_\mathcal{H}\mathcal{M}_{-}^{\alpha }=K_\alpha\thinspace_\mathcal{H}\mathcal{I}_{-}^{(\alpha-1)/2 }$, it is easy to check that  \begin{equation}g(x):=
      \left( \thinspace _{\mathcal{H}}\mathcal{D}_{-}^{(\alpha-1)/2 }1_{(a,b)}\right) (x) \in L_{2}(\mathbb{R}^+), \label{l2}\end{equation} 
   since
      \begin{equation}
\left( _\mathcal{H}\mathcal{D}_{-}^{(\alpha-1)/2 }1_{(a,b)}\right) (x)=\frac{1}{\Gamma
((3-\alpha)/2 )}\left[ \left( \log \frac{b}{x}\right) _{+}^{(1-\alpha)/2 }-\left(
\log \frac{a}{x}\right) _{+}^{(1-\alpha)/2}\right] ,  \notag
\end{equation}%
belongs to $ L_{p}(\mathbb{R}^{+}),$ as far as $p<2/(\alpha-1).$ Let now denote by $\mathcal{I}_\beta(L_2(\mathbb{R}^+))$, for $\beta>0,$ the class of functions that can be represented as Hadamard integrals, i.e. the class of functions $f$ such that $f=\thinspace_\mathcal{H}\mathcal{I}^\beta_-\varphi$, for some $\varphi \in L_2(\mathbb{R}^+) $.  It can be checked  that if 
\begin{equation}\label{zero-1}
\left(\thinspace_\mathcal{H}\mathcal{I}^\beta_-\varphi \right)(x)=\frac{1}{\Gamma(\beta)}\int_{ x}^\infty u^{-1}\left(\log \frac{u}{x}\right)^{\beta -1}  \varphi(u)du=0, \qquad  a.a. \thickspace  x \in \mathbb{R}^+, 
\end{equation}
for $\beta \in (0,1)$,  then $\varphi(x)=0,$ $a.a. \thickspace  x \in \mathbb{R}^+$, so that the uniqueness of such function $\varphi$ is guaranteed.  Indeed, by integrating \eqref{zero-1} from any $a>0$ to $\infty$ with weight $\left(\log \left(x/a\right)\right)^{-\beta}x^{-1}$, we get that \begin{eqnarray}\int_a^\infty \left(\log \frac xa\right)^{-\beta}
\int_x^\infty \frac{\varphi(u)}{u}\left(\log \frac ux\right)^{\beta-1}du\frac{dx}{x} &=&\int_a^\infty\frac{\varphi(u)}{u}\int_a^u\left(\log \frac ux\right)^{\beta-1}\left(\log \frac xa\right)^{-\beta}\frac{dx}{x} du \notag \\
&=&\frac{\pi}{\sin(\pi\beta)}\int_a^\infty\frac{\varphi(u)}{u}du=0. \label{zero-2} \end{eqnarray}
The interchange of integration's order can be justified by the Fubini-Tonelli theorem, because $$\int_a^\infty\frac{|\varphi(u)|}{u}du \le \frac 1a\|\varphi\|_{L_2}<\infty. $$ Thus, 
from \eqref{zero-2}, we immediately get that $\varphi(u)=0 $, a.e. $u\in \mathbb{R}^+$. 

By considering Lemma \ref{leminverse} (for $\beta=(\alpha-1)/2 \in (0,1)$), we get that \[1_{(a,b)}(x)=K\left(\thinspace_\mathcal{H}\mathcal{I}_{-}^{(\alpha-1)/2 }g\right)(x) \in
\mathcal{I}_\beta(L_2(\mathbb{R}^+)),\] for $\beta=(\alpha-1)/2$, and this
is true also for more general step functions. Since the class of step functions is dense in $L_2(\mathbb{R}^+)$, it follows that also $\mathcal{I}_\beta(L_2(\mathbb{R}^+))$ is dense in $L_2(\mathbb{R}^+)$. Then, we can apply Lemma \ref{lemHL} (with $\gamma=1/p$ and $p=2$): for any $h \in\mathcal{I}_\beta(L_2(\mathbb{R}^+))  $, with $h=\thinspace_\mathcal{H}\mathcal{I}_{-}^{(\alpha-1)/2 }g$ and $g \in L_2(\mathbb{R}^+),$ there exists a sequence of step functions $\left\{g^*_n\right\}_{n \geq 1} \in L_2(\mathbb{R}^+)$, with $g^*_n \to g$, as $n \to \infty$, such that 
 \[
 \Vert \thinspace_\mathcal{H}\mathcal{I}_{-}^{(\alpha-1)/2 }g^*_n-h \Vert_{L_2} \leq 2^{(\alpha-1))/2}\Vert g^*_n-g \Vert_{L_2} \to 0, \qquad n \to \infty.
 \]     
As a consequence, the linear span of $\left\{_{\mathcal{H}}\mathcal{M}^\alpha_- 1_{(a,b)}; a,b \in \mathbb{R}^+ , a<b\right\}$ is dense in $\mathcal{I}_\beta(L_2(\mathbb{R}^+))$, for $\beta=(\alpha-1)/2$, and thus it is dense also in $L_2(\mathbb{R}^+)$.  
    \end{proof}
By the linearity of $_{\mathcal{H}}\mathcal{M}^\alpha_-$ and by Theorem \ref{thmden}, for a function $f \in L_{2,\alpha}(\mathbb{R}^+)$ and a sequence of step functions $\left\{ f^*_n \right\}_{n \geq 1}$ in $L_{2,\alpha}(\mathbb{R}^+)$ such that $\lim_{n \to \infty} f^*_n=f$, we have that $\lim_{n \to \infty} \;_{\mathcal{H}}\mathcal{M}^\alpha_-f^*_n= \; _{\mathcal{H}}\mathcal{M}^\alpha_-f$. Therefore we can write that 
    \begin{equation}
    I_\alpha(f)=\lim_{n \to \infty}\int_{\mathbb{R}^+}\left(\;_{\mathcal{H}}\mathcal{M}_-^\alpha f^*_n\right)(s)dB(s)=\lim_{n \to \infty}\int_{\mathbb{R}^+} f^*_n(s)dB^\mathcal{H}_\alpha(s), \label{ialfa}
    \end{equation}
so that the integral in Def. \ref{defibeta} is well-defined.

\subsection{Riemann-Stieltjes  integration  of smooth integrands with respect to the $\mathcal{H}$-fBm} In Subsection \ref{integinteg}  the general construction of the Wiener integral is provided. However, we can take into account H\"{o}lder properties of $\mathcal{H}$-fBm, established in Section \ref{pathwise} and construct the integral w.r.t. $\mathcal{H}$-fBm as the limit of Riemann sums, for sufficiently smooth deterministic integrands.  So, the following   statement  is a consequence of  the H\"{o}lder properties of $B^\mathcal{H}_\alpha$.
\begin{theorem}\label{cor:si}
Let $T>0$; then the integral
\[\int_0^T f(s)dB^\mathcal{H}_\alpha(s)\] 
exists as  the a.s. limit of Riemann-Stieltjes sums  if
 $f:[0,T] \to \mathbb{R}$ is  H\"older continuous up to order $\beta$ such that
 \begin{enumerate}[label=(\roman*)]
    \item if $\alpha \in (0,1)$, then $\beta>1-\frac{\alpha}{2}$,  
     \item if  $\alpha \in (1,2)$, then  $\beta>\frac{1}{2}$.
\end{enumerate}
Moreover, for any $\alpha \in (0,1) \cup (1,2)$ and for any function $f \in \mathcal{C}^{(1)}([0,T])$, the following integration-by-parts formula holds true
\[\int_0^T f(s)dB^\mathcal{H}_\alpha(s)=f(T)B^\mathcal{H}_\alpha(T)-\int_0^T B^\mathcal{H}_\alpha(s)f'(s)ds.\]
   
\end{theorem}
\begin{proof}
    The proof follows immediately from Theorem \ref{thm1} and Theorem 4.2.1  in \cite{ZAH}, where it is proved that such integrals exist as the limits of the Riemann-Stieltjes sums if the integrand and integrator are H\"older continuous of orders $a$ and $b$ with $a+b>1$.
\end{proof}
 \begin{rem} (i) If the integrand $f$ is deterministic, then, obviously, $\int_0^T f(s)dB^\mathcal{H}_\alpha(s)$ is a Gaussian random variable. 
 
 (ii) If necessary, one can consider the extension of Theorem  \ref{cor:si} to random integrands $f$ having a.s.  the same H\"{o}lder property.
  \end{rem}
 \begin{rem} Let $T>0$ and   $\alpha\in(1,2)$. In this case we can integrate representation \eqref{mvn} by parts and get that 
 $$ B^\mathcal{H}_\alpha (t)=\frac{\alpha-1}{2\sqrt{\Gamma(\alpha)}}\int_{0}^{t} B(s)\left( \log \frac{t}{s}\right)^{(\alpha-3)/2}\frac{ds}{s}=\frac{\alpha-1}{2\sqrt{\Gamma(\alpha)}}\int _{0}^{1} B(ts)\left( \log \frac{1}{s}\right)^{(\alpha-3)/2}\frac{ds}{s}.  $$ Now, let $f:[0,T]\to \mathbb{R}$ be a continuous function, for which we know that  the integral
 $\int_0^T f(s)dB^\mathcal{H}_\alpha(s)$ 
exists as a limit  of Riemann-Stieltjes sums for any sequence of partitions with vanishing diameter. Then 
\begin{align}\label{stochint}&\int_0^T f(s)dB^\mathcal{H}_\alpha(s)\notag\\
&=\frac{\alpha-1}{2\sqrt{\Gamma(\alpha)}}\lim_{N\to \infty}\int_0^1\left(\sum_{k=0}^{N-1}f(Tk/N)(B(T(k+1)s/N)-B(Tks/N))\right)\left( \log \frac{1}{s}\right)^{(\alpha-3)/2}\frac{ds}{s}.\end{align}  Obviously,  
  \[ \sum_{k=0}^{N-1}f(Tk/N)(B(T(k+1)s/N)-B(Tks/N))\to \int_0^Tf(u)dB(su)  \] in $L_2(\Omega)$.
  Moreover, using  the Cauchy-Schwartz inequality and isometry $$\mathbb{E}\left(\sum_{k=0}^{N-1}f(Tk/N)(B(T(k+1)s/N)-B(Tks/N))\right)^2=\sum_{k=0}^{N-1}f^2(Tk/N)\frac{Ts}{N},$$  we can write that
  \begin{eqnarray}
&&\mathbb{E}\left(\int_0^1\left(\sum_{k=0}^{N-1}f(Tk/N)(B(T(k+1)s/N)-B(Tks/N))\right)\left( \log \frac{1}{s}\right)^{(\alpha-3)/2}\frac{ds} {s}\right)^2 \notag\\ &\le&  \int_0^1\mathbb{E}\left(\sum_{k=0}^{N-1}f(Tk/N)(B(T(k+1)s/N)-B(Tks/N))\right)^2\left( \log \frac{1}{s}\right)^{(\alpha-3)/2}\frac{ds} {s^{3/2}}  \notag\\
&\times& \int_0^1\left( \log \frac{1}{s}\right)^{(\alpha-3)/2}\frac{ds} {s^{1/2}} \notag  \leq \max_{t\in[0,T]} f^2(t) T \left( \int_0^1\left( \log \frac{1}{s}\right)^{(\alpha-3)/2}\frac{ds} {s^{1/2}}\right)^2 \\
   &\le&  \max_{t\in[0,T]} f^2(t) T\left(\int_0^\infty v^{(\alpha-3)/2}e^{-v/2}dv\right)^2. \notag
      \end{eqnarray}
     
      It means that the sequence of integrals in the right-hand side of \eqref{stochint} is uniformly integrable, whence 

     \begin{align*}\int_0^T f(s)dB^\mathcal{H}_\alpha(s)&=\frac{\alpha-1}{2\sqrt{\Gamma(\alpha)}}\int_0^1\int_0^Tf(u)dB(su)\left( \log \frac{1}{s}\right)^{(\alpha-3)/2}\frac{ds} {s}\\
     &=\frac{\alpha-1}{2\sqrt{\Gamma(\alpha)}}\int_0^T\int_{u/T}^1f(u/s)\left( \log \frac{1}{s}\right)^{(\alpha-3)/2}\frac{ds} {s}dB(u)\\
     &=\frac{\alpha-1}{2\sqrt{\Gamma(\alpha)}}\int_{0}^T\int_{u}^Tf(Tu/s)\left( \log \frac{T}{s}\right)^{(\alpha-3)/2}\frac{ds} {s}dB(u). 
      \end{align*}
     Then, for any continuous $f,g:[0,T]\to \mathbb{R}$, which are Riemann-integrable (i.e., whose integrals are the limits of their limits of integral sums), we have that \begin{align*}
     &\mathbb{E}\left(\int_0^T f(s)dB^\mathcal{H}_\alpha(s)\int_0^T g(r)dB^\mathcal{H}_\alpha(r)\right)\\
     &= \frac{(\alpha-1)^2}{4\Gamma(\alpha)} 
     \int_0^T\int_0^Tf(\rho)g(\nu)(\rho\nu)^{-1}\int_0^{\rho\wedge\nu}\left(\log\frac{\rho}{u}\log\frac{\nu}{u}\right)^{(\alpha-3)/2}du d\rho d\nu. \end{align*}
     The latter equality means that we find the form of the  kernel for integration w.r.t. $\mathcal{H}$-fBm in the case $\alpha\in(1,2)$, and this kernel equals $$K(\rho, \nu)=(\rho\nu)^{-1}\int_0^{\rho\wedge\nu}\left(\log\frac{\rho}{u}\log\frac{\nu}{u}\right)^{(\alpha-3)/2}du =(\rho\nu)^{-1}\sigma_{\rho,\nu}^{(\alpha-3)/2}  . $$ 
     Just to compare, recall that for fBm with Hurst index $H\in(1/2,1) $ the kernel equals 
     $$K(\rho, \nu)=H(2H-1)|\rho-\nu|^{2H-2}.$$
 \end{rem}
   \begin{rem} Let us introduce the  Hadamard fractional Ornstein-Uhlenbeck process ($\mathcal H$-fOU) $\{X_\alpha^{\mathcal H}(t)\}_{t\geq 0}$ satisfying the Langevin stochastic equation
 $$X_\alpha^{\mathcal H}(t)=  x_0-\theta\int_0^tX_\alpha^{\mathcal H}(s)ds +B_\alpha^{\mathcal H}(t),\quad X_\alpha^{\mathcal H}(0)=x_0, $$
 where $\theta>0$ and $\alpha\in(0,1)\cup (1,2).$ This process has been studied in a more general setting (involving Le Roy measures) in \cite{BEG}. The authors shown that the solution process is given by
$$X_\alpha^{\mathcal H}(t)=x_0e^{-\theta t} +B_\alpha^{\mathcal H}(t)-\theta e^{-\theta t}\int_0^tB_\alpha^{\mathcal H}(s)e^{\theta s}ds.$$ 
Since the exponential function  is Lipschitz continuous on $[0,t]$, we can apply the integration by-parts formula to $\int_0^te^{\theta s} d B_\alpha^{\mathcal H}(s) $ and get that 
$$X_\alpha^{\mathcal H}(t)=e^{-\theta t} x_0+e^{-\theta t}\int_0^te^{\theta s} d B_\alpha^{\mathcal H}(s).$$
 Obviously,  $X_\alpha^{\mathcal H}(t)$ is a Gaussian random variable with mean $ x_0 e^{-\theta t}$ and variance $$t-2\theta\int_0^t\sigma^\alpha_{t,s}e^{\theta(s-t)}ds+\theta^2\int_0^t\int_0^t\sigma^\alpha_{u,s}e^{\theta(u+s-2t)}dsdu,$$ where covariance function $\sigma^\alpha_{t,s}$ was introduced in \eqref{varGadamar}.  
 \end{rem}

\subsection{Inverse representation and multiplicative Sonine pairs}

Now we prove that the  representation \eqref{mvn} can be inverted.
    \begin{theorem}\label{thmbt}
    Let $\alpha \in (0,1)\cup(1,2)$ and let
    \begin{equation}\label{defm}
\left( \thinspace ^\mathcal{H}\mathcal{\overline{M}}_{-}^{\alpha }f\right) (x):=
\begin{cases}
 \left( \thinspace _{\mathcal{H}}\mathcal{I}_{-}^{(1-\alpha)/2 }f\right) (x),& \alpha \in (0,1),\\
f(x),& \alpha =1,\\
 \left( \thinspace _{\mathcal{H}}\mathcal{D}_{-}^{(\alpha-1)/2 }f\right) (x),& \alpha \in (1,2).
\end{cases}
\end{equation}
Then $_\mathcal{H}\mathcal{\overline{M}}_{-}^{\alpha }1_{[0,t)} \in L_2(\mathbb{R}^+)$ and the following inverse representation holds
    \begin{equation}
      B(t)=C'_\alpha\int _{0}^{\infty} \left( \thinspace _\mathcal{H}\mathcal{\overline{M}}_{-}^{\alpha }1_{[0,t)}\right)(s)dB^\mathcal{H}_\alpha(s), \label{mvn2}
  \end{equation}
  where $C_\alpha':=\sqrt{\Gamma(\alpha)}/\Gamma ((\alpha+1)/2)$.

    \end{theorem}
    \begin{proof}
    We have already seen, in the Theorem \ref{thmden}, that$_\mathcal{H}\mathcal{\overline{M}}_{-}^{\alpha }1_{[0,t)} \in L_{2}(\mathbb{R}^+)$,  for $\alpha \in (1,2)$ (see \eqref{l2}), while, for $\alpha \in (0,1)$, it follows by Lemma \ref{HL}, with $\gamma=1/p$ and $p=2$.
    
    Moreover, by considering Lemma \ref{leminverse} and \eqref{ialfa}, we can write, for $\alpha \in (0,1)$, that
    \begin{eqnarray}
    B(t)&=&\int_0^\infty 1_{[0,t)}(s)dB(s)=\int_0^\infty\left(_{\mathcal{H}}\mathcal{D}_{-}^{(1-\alpha)/2 } \; _{\mathcal{H}}\mathcal{I}_{-}^{(1-\alpha)/2 } 1_{[0,t)}\right)(s)dB(s) \notag \\
&=&\frac{1}{K_\alpha }\int_0^\infty\left(_{\mathcal{H}}\mathcal{M}_{-}^{\alpha } \; _{\mathcal{H}}\overline{\mathcal{M}}_{-}^{\alpha } 1_{[0,t)}\right)(s)dB(s)\notag \\
&=& \frac{1}{K_\alpha}\int_0^\infty\left( _{\mathcal{H}}\overline{\mathcal{M}}_{-}^{\alpha } 1_{[0,t)}\right)(s)dB_\alpha^\mathcal{H}(s),\notag 
    \end{eqnarray}
    where $K_\alpha$ is given in \eqref{malf}. Thus \eqref{mvn} follows, by denoting $C_\alpha':=K_\alpha ^{-1}$. 
    By analogous steps, we get the same result for $\alpha \in (1,2)$.
    \end{proof}

\begin{coro}\label{corobt}
    The following inverse representation holds for  $\mathcal{H}$-fBm, for any $t \geq 0$ and for $\alpha \in (0,1) \cup(1,2)$:
    \begin{equation}
      B(t)=\frac{\sqrt{\Gamma(\alpha)}}{{ \Gamma((\alpha+1)/2)}\Gamma((3-\alpha)/2)}\int _{0}^{t} \left( \log \frac{t}{s}\right)^{(1-\alpha)/2}dB^\mathcal{H}_\alpha(s). \label{inv1}
  \end{equation}

\end{coro}
\begin{proof}
For $\alpha \in (0,1)$ and $0 \leq s<t$, we recall the definition \eqref{intmin}, for $\mu=0$, so that we get
\begin{align*}
&\left( _{\mathcal{H}}\mathcal{I}_{-}^{(1-\alpha)/2 }1_{[0,t)}\right) (s)=\frac{1}{\Gamma
((1-\alpha)/2 )}\int_{s}^{t}\left( \log \frac{z}{s}\right) ^{-(1+\alpha)/2
}\frac{dz}{z}=\frac{1}{\Gamma ((3-\alpha)/2 )}\left( \log \frac{t}{s}\right)_+
^{(1-\alpha)/2 }.
\end{align*}

Analogously, for $\alpha \in (1,2)$ and $0 \leq s<t$, we can write, in view of \eqref{dermin}, with $\mu=0$, that

\begin{align*}
&\left( _{\mathcal{H}}\mathcal{D}_{-}^{(\alpha-1)/2 }1_{[0,t)}\right) (s)=-s \frac{d}{ds}\left(\thinspace_{\mathcal{H}}\mathcal{I}_{-}^{(3-\alpha)/2 }1_{[0,t)}\right)(s)\\
&= -\frac{s}{\Gamma
((3-\alpha)/2 )}\frac{d}{ds}\int_{s}^{t}\left( \log \frac{z}{s}\right) ^{(1-\alpha)/2
}\frac{dz}{z}=\frac{1}{\Gamma ((3-\alpha)/2 )}\left( \log \frac{t}{s}\right)_+
^{(1-\alpha)/2 },
\end{align*}
which could be alternatively derived in view of Lemma \ref{leminverse}.
Then, by \eqref{mvn2}, we can write, for $\alpha \in (0,1)$ that:
 \begin{eqnarray}
      B(t)&=&C'_\alpha\int _{0}^{\infty} \left( \thinspace _\mathcal{H}\mathcal{\overline{M}}_{-}^{\alpha }1_{[0,t)}\right)(s)dB^\mathcal{H}_\alpha(s)=\frac{1}{K_\alpha}\int _{0}^{\infty} \left( \thinspace _\mathcal{H}\mathcal{I}_{-}^{(1-\alpha)/2 }1_{[0,t)}\right)(s)dB^\mathcal{H}_\alpha(s), \notag\\
      &=& \frac{1}{K_\alpha\Gamma((3-\alpha)/2)}\int _{0}^{\infty} \log \left(\frac{t}{s}\right)^{(1-\alpha)/2}dB^\mathcal{H}_\alpha(s), 
  \end{eqnarray}
  which gives \eqref{inv1}, by recalling that $K_{\alpha}=\Gamma ((\alpha+1)/2)/\sqrt{\Gamma(\alpha)}$, and analogously, for $\alpha \in (1,2)$. 



\end{proof}


 \begin{remark} Let us look at the situation with the inverse representation \eqref{inv1} for $\alpha\in(1,2)$ from the point of view of ``multiplicative Sonine  pairs''. It follows from \eqref{mvn} and \eqref{inv1} that, for any $z>0$,
\begin{eqnarray}B^\mathcal{H}_\alpha(z)&=&\frac{1}{\sqrt{\Gamma(\alpha)}}\int_0^z\left(\log\frac{z}{t}\right)^{\frac{\alpha-1}{2}}dB(t) \notag \\&=&\frac{1}{ \Gamma((\alpha+1)/2)\Gamma((3-\alpha)/2)}\int_0^z\left(\log\frac{z}{t}\right)^{\frac{\alpha-1}{2}}d_t\left(\int _{0}^{t} \left( \log \frac{t}{s}\right)^{(1-\alpha)/2}dB^\mathcal{H}_\alpha(s)\right).\notag
     \end{eqnarray}
    
      Now, for $\alpha>1$ we can integrate by parts, and get that
     \begin{equation}\begin{gathered}\label{identity}B^\mathcal{H}_\alpha(z)= \frac{\alpha-1}{2}\frac{1}{ \Gamma((\alpha+1)/2)\Gamma((3-\alpha)/2)}\int_0^z\left(\log\frac{z}{t}\right)^{\frac{\alpha-3}{2}}\frac1t  \left(\int _{0}^{t} \left( \log \frac{t}{s}\right)^{(1-\alpha)/2}dB^\mathcal{H}_\alpha(s)\right)dt\\= \frac{1}{\Gamma((\alpha-1)/2)\Gamma((3-\alpha)/2)}\int_0^z\left(\int_s^z \left(\log\frac{z}{t}\right)^{\frac{\alpha-3}{2}}\left( \log \frac{t}{s}\right)^{(1-\alpha)/2}\frac{dt}{t}\right)dB^\mathcal{H}_\alpha(s). 
     \end{gathered}
      \end{equation}
     We observe the effect that logarithms play the role of ``multiplicative Sonine  pairs'',  because 
      \begin{equation}\label{multSonine}\begin{gathered} \int_s^z \left(\log\frac{z}{t}\right)^{\frac{\alpha-3}{2}}\left( \log \frac{t}{s}\right)^{(1-\alpha)/2}\frac{dt}{t}=\int_{\log s}^{\log z}(\log z-u)^{\frac{\alpha-3}{2}}(u-\log s)^{(1-\alpha)/2}du\\=B\left(\frac{\alpha-1}{2},\frac{3-\alpha}{2}\right)=\Gamma\left(\frac{\alpha-1}{2}\right)\Gamma\left(\frac{3-\alpha}{2}\right )=\frac{\pi}{\cos(\pi\alpha/2)},
      \end{gathered}
      \end{equation}
and the latter equality, indeed, transforms \eqref{identity} into identity. It is interesting to compare this situation with the kernels participating in the representation of fractional Brownian motion, where the Sonine pair $c(s) = s^{-\alpha}$ and $h(s) = s^{\alpha-1}$ with some $\alpha\in(0, 1/2)$ plays the crucial role, see, for example, discussion and some other examples of Sonine pairs in \cite{MiShSh}.
 \end{remark} 

\section{Reproducing kernel Hilbert space and  law of iterated logarithm for  $\mathcal{H}$-fBm}\label{reprod}
Having in hands  ``multiplicative Sonine  pairs'', we can proceed with reproducing kernel Hilbert spaces and  law of iterated logarithm.  By the reasons that will be clear later, we consider only   $\alpha\in(0,1).$
 Taking into account the general theory given in \cite{aron} and the results in Section 2.2 of \cite{Hult},  the reproducing kernel Hilbert space of $\{B^\mathcal{H}_\alpha(t), t\ge 0\}$ can   be represented as the image of
$L_2(\mathbb{R})$ under the integral transform
\begin{equation}\label{rkhs}F(t)=(Af)(t)=\frac{1}{\sqrt{\Gamma(\alpha)}}\int_0^t\left(\log\frac{t}{s}\right)^{\frac{\alpha-1}{2}}f(s)ds, \qquad t\ge 0. 
    \end{equation}
    Multiplying both sides of \eqref{rkhs} by $\left(\log \frac{z}{t}\right)^{\frac{-\alpha-1}{2}}\frac1t dt$, integrating from $0$ to $z$ and using the ``multiplicative Sonine  pairs'' property \eqref{multSonine} with powers $\frac{\alpha-1}{2}$ and $\frac{-\alpha-1}{2}$, we obtain that $$\int_0^z\left(\log \frac{z}{t}\right)^{\frac{-\alpha-1}{2}}\frac{F(t)}{t} dt=\frac{\pi}{\sqrt{\Gamma(\alpha)}\cos\left(\frac{\pi\alpha}{2}\right)}\int_0^zf(s)ds,$$
    whence
    $$f(z)=(A^{-1}F)(z)=\frac{\sqrt{\Gamma(\alpha)}\cos\left(\frac{\pi\alpha}{2}\right)}{\pi}\frac{d}{dz}\left(\int_0^z\left(\log \frac{z}{t}\right)^{\frac{-\alpha-1}{2}}\frac{F(t)}{t} dt\right), \qquad a.e.$$ Note that the evident restriction $\frac{-\alpha-1}{2}>-1$ implies that we can consider here only $\alpha\in (0,1)$.
    Summarizing, our RKHS has a form $\mathcal{H}=\{F=Af, f\in L_2(\mathbb{R})\}$ with the inner product $$\langle F,G\rangle_{\mathcal{H}}=\langle (A^{-1}F), (A^{-1}G)\rangle_{L_2(\mathbb{R})}.$$
    \begin{rem} Similarly, we can consider, for any $T>0$, the RKHS of the process  $\{B^\mathcal{H}_\alpha(t), t\in[0,T]\}$, replacing $L_2(\mathbb{R})$ with $L_2([0,T])$. 
        
    \end{rem}
    
    Now, in order to consider the law of iterated logarithm for the process  $B^\mathcal{H}_\alpha$ both at zero and at $\infty,$  we apply Corollary 3.1 from  \cite{arcones} and Theorem 1 from \cite{Lai} and the results from \cite{oodaira}, respectively (see the Appendix in Section \ref{app} ). 
    \begin{theorem} Let $\alpha\in(0,1)$. \begin{itemize}
    \item[(i)] For any $T>0$,  with probability one, the set 
    $$\left\{ \frac{B^\mathcal{H}_\alpha(ut)}{u^{1/2}(2 \log \log u^{-1})^{1/2}},\,t\in[0,T]\right\}$$
    is relatively compact in $l_{\infty}([0,T])$, as $u\to 0+$, and its   set of limit points, as $u\to 0+$, is the unit ball   of
the RKHS of   $\{B^\mathcal{H}_\alpha(t),\,  t\in[0,T]\}$. 
 \item[(ii)] For any $T>0$,  with probability one, the set  
    $$\left\{ \frac{B^\mathcal{H}_\alpha(nt)}{n^{1/2}(2 \log \log n^{-1})^{1/2}},\,t\in[0,T]\right\},$$
    is relatively compact in $C([0,T])$, and its set of limit points, as $n\to\infty$, coincides with   the unit ball  of
the RKHS of   $\{B^\mathcal{H}_\alpha(t),  t\in[0,T]\}$. 
    \end{itemize} 
        
    \end{theorem}
    \begin{proof} 
    \begin{itemize}
    \item[(i)]
    It is necessary to check the conditions of Theorem \ref{loglog1}. Obviously, condition (i) is fulfilled with $\gamma=1/2$, conditions (ii) and (iii) are evident. Condition (iv):    taking into account that $\alpha\in(0,1)$, we have for any $s,t\in[0,T]$ and $u<t/s$: \begin{equation*} \begin{gathered} u^{-1/2} \sigma^\alpha_{t,us}= u^{-1/2}\int_0^{us}\left(\log\frac{t}{v}\right)^{\frac{\alpha-1}{2}}\left(\log\frac{us}{v}\right)^{\frac{\alpha-1}{2}}dv\\=u^{1/2}\int_0^{s}\left(\log\frac{t}{uz}\right)^{\frac{\alpha-1}{2}}\left(\log\frac{s}{z}\right)^{\frac{\alpha-1}{2}}dz\le u^{1/2}\int_0^{s} \left(\log\frac{s}{z}\right)^{ \alpha-1 }dz \to 0\end{gathered}\end{equation*} as $u\to 0+$, whence condition (iv) of Theorem \ref{loglog1} follows. 
    
       \item[(ii)] We check the conditions of Theorem \ref{loglog2}. According to Theorem  \ref{thm1}, we should  put in condition (v), (a)  $g_T(h)=C_{T,\alpha}h^{\alpha}.$ Then condition (v), (b) is evident. Furthermore, $R(T,T)=T$ while $g_T(1)=C_{T,\alpha}\le CT^{1-\alpha}$, according to relation \eqref{const}, whence (v), (c) follows. In condition (vi) function $v(r)=r$, therefore, (vi) is fulfilled, while in condition (vii) $Q(t,v)=\left(\log\frac{t}{v}\right)^{\frac{\alpha-1}{2}}$, therefore, $u(r)=1$, condition (vii) is also fulfilled, and the proof follows. 
    \end{itemize}    
    \end{proof}

 \section{Appendix}\label{app}

 \subsection{Auxiliary results}
 Let us prove a result that is essential in order to produce  upper bounds for the incremental variance.
 \begin{lem}\label{implem}
    Let $0<s\le t$. Then for any $p\ge 1$ the following inequality holds:
\begin{equation}\label{auxil-e-ineq}s\left(\log\frac ts\right)^p\le L_p(t-s),  \end{equation}
  where $L_p=\left(p(p-1)^{p-1}e^{1-p}\right)\vee 1$.
 \end{lem}
\begin{proof}  Let $p>1$. Dividing both parts of \eqref{auxil-e-ineq} by $s>0$ and denoting $x=\frac ts\ge 1$, we get that it is enough to establish the following inequality:
\begin{equation}\label{auxil-e-ineq1} \left(\log x\right)^p\le L_p(x-1), \qquad x\ge 1. \end{equation}
Obviously, \eqref{auxil-e-ineq1} holds for $x=1$, being an equality; therefore, it is enough to compare the derivatives, establishing inequality
 \begin{equation*}\label{auxil-e-ineq2}
  \frac{(\log x)^{p-1}}{x}\le \frac{L_p}{p}, \qquad x\ge 1.
 \end{equation*}
The function $ (\log x)^{p-1}/x$ is zero at 1 and tends to 0 at $\infty$, achieving its maximum at $x=e^{p-1}$; this maximal value equals $(p-1)^{p-1}e^{1-p}$, whence the inequality \eqref{implem} follows. For $p=1$ the proof is evident.
  
\end{proof}
Now we formulate two theorems about the law of iterated logarithm for Gaussian processes. The first result can be obtained as in  Corollary 3.1 in \cite{arcones}, a bit  adapted  ad simplified for our purposes. We consider Theorems \ref{loglog1} and \ref{loglog2} as a couple, therefore continue numeration of conditions in them. 
\begin{theorem}\label{loglog1}
 Let $\{X(t),  t\in[0,T]\}$ be a  zero mean Gaussian process and let 
$\gamma> 0$.
Suppose that the following conditions are satisfied:
\begin{itemize}
\item[(i)] $ \mathbb{E}[X(ut)X(us)] = u^{2\gamma}
\mathbb{E}[X(t)X(s)]$ for each $ 0< u< 1$ and each $s, t\in[0,T]$.

\item[(ii)] $\sup_{t\in[0,T]} |X(t)| < \infty$ a.s.

\item[(iii)] $\lim_{u\to 1-} \mathbb{E}[X(ut)X(t)] = \mathbb{E}[X^2(t)]$  for each $t\in[0,T].$

\item[(iv)]
$\lim \sup_{u\to 0} u^{-\gamma}|\mathbb{E}[X(t)X(us)]|= 0,\,s,t\in[0,T].$
\end{itemize}
Then,  
with probability one, the set 
    $$\left\{ \frac{X(ut)}{u^{1/2}(2 \log \log u^{-1})^{1/2}},\,t\in[0,T]\right\}$$
    is relatively compact in $l_{\infty}([0,T])$, as $u\to 0+$, and its   set of limit points, as $u\to 0+$, is the unit ball   of
the RKHS of   $\{X(t), t\in[0,T]\}$. 
    
\end{theorem}
The next result is the combination of Theorem 1 from \cite{Lai}, Theorem 4 from \cite{oodaira}, and the result formulated in Section 3  from \cite{Lai}.
\begin{theorem}\label{loglog2}
Let $\{X(t), t\ge 0\}$ be  a separable real-valued zero-mean Gaussian process
with  continuous covariance function $R (s, t) $ satisfying  the conditions 
\begin{itemize}
    \item [(v)] (a) For any $T > 0$, there exists a continuous non-decreasing function $g_T(h)$  such
that for all $t, t+h\in [0, T]$,
$$\mathbb{E}[ X(t+h)-X(t)]^2\le g_T(h)\to 0$$ as $h\to 0$,

(b)
  $$(g_T(1))^{-1/2}\int_1^\infty (g_T(e^{-u^2} ))^{1/2}du\le C<\infty,$$
 and 
 
 (c) $R(T, T)/g_T(1)\uparrow\infty$ as $T\to\infty$.
\item[(vi)] There exists a positive function $v(r), r>0$, such that
$v(r)\uparrow\infty$ as $r\to\infty$ and $$R(rs, rt)=v(r)R(s,t)$$ for all $r>0, s,t\ge 0.$
\item[(vii)] $R(s, t)$  has a representation of the form $R(s, t)=\int_0^{s\wedge t}Q(t,v)Q(s, v)dv, s, t>0,$
 where $\int_0^tQ^2(t, v)dv<\infty$ for all $t>0$, and there exists a function $u(r)$ such that
 
$Q (r t, r v) = u (r) Q (t, v)$  for all $r > 0, t, v > 0$ and $ru^2(r)\uparrow\infty$ as $r\to\infty$.
\end{itemize}
Then, for any $T>0$,  with probability one, the set  
    $$\left\{ \frac{X(nt)}{n^{1/2}(2 \log \log n^{-1})^{1/2}},\,t\in[0,T]\right\},$$
    is relatively compact in $C([0,T])$, and its set of limit points, as $n\to\infty$ coincides with   the unit ball  of
the RKHS of   $\{X(t),  t\in[0,T]\}$.
\end{theorem}

\section*{Acknowledgments}
L.B. acknowledges financial support under NRRP, by the Italian Ministry of University and Research, funded by NextGenerationEU – Project Title ``Non–Markovian Dynamics and Non-local Equations'' – 202277N5H9 - CUP: D53D23005670006 - Grant No. 973, 2023.

A.D.G is partially supported by  Italian Ministry of University and Research (MUR) under PRIN 2022 (APRIDACAS),
Anomalous Phenomena on Regular and Irregular Domains: Approximating Complexity for the Applied Sciences, Funded by EU - Next Generation EU
CUP B53D23009540006 - Grant Code 2022XZSAFN - PNRR M4.C2.1.1.

Y.M. is supported by The Swedish Foundation for Strategic Research, grant UKR24-0004, and by the Japan Science and Technology Agency CREST, project reference number JPMJCR2115.

\end{document}